\RequirePackage{snapshot}
\documentclass[final]{siamltex}   %

\usepackage{amssymb}
\usepackage{microtype}
\usepackage{subdepth}  %
\usepackage{mathtools}
\usepackage{xspace}
\usepackage{boxedminipage}
\usepackage{graphicx}
\usepackage[svgnames]{xcolor}
\usepackage{colortbl}

\usepackage[longnamesfirst]{natbib}
\bibpunct{(}{)}{;}{a}{,}{,}

\usepackage{hyperref}
\hypersetup{
  breaklinks,
  colorlinks,
  citecolor={DarkBlue},
  urlcolor={DarkBlue},
  linkcolor={DarkBlue}
}

\usepackage{amsmath}
\usepackage{amsfonts}
\usepackage[text={5.125in,8.25in},centering]{geometry}

\RequirePackage{amssymb}
\RequirePackage{subdepth}
\makeatletter
\newcounter{parenttheorem}%

\makeatother

\makeatletter
\newcommand\RedeclareMathOperator{%
  \@ifstar{\def\rmo@s{m}\rmo@redeclare}{\def\rmo@s{o}\rmo@redeclare}%
}
\newcommand\rmo@redeclare[2]{%
  \begingroup \escapechar\m@ne\xdef\@gtempa{{\string#1}}\endgroup
  \expandafter\@ifundefined\@gtempa
     {\@latex@error{\noexpand#1undefined}\@ehc}%
     \relax
  \expandafter\rmo@declmathop\rmo@s{#1}{#2}}
\newcommand\rmo@declmathop[3]{%
  \DeclareRobustCommand{#2}{\qopname\newmcodes@#1{#3}}%
}
\@onlypreamble\RedeclareMathOperator
\makeatother

\newcommand{\http}[1]{\href{http://#1}{\texttt{\nolinkurl{#1}}}}

\newcommand{\mailto}[1]{\href{mailto:#1}{\texttt{#1}}}

\def\BibTeX{{\rm B\kern-.05em{\sc i\kern-.025em b}\kern-.08em
    T\kern-.1667em\lower.7ex\hbox{E}\kern-.125emX}}

\def\minim{\mathop{\hbox{minimize}}}
\def\maxim{\mathop{\hbox{maximize}}}
\def\minimize#1{\displaystyle\minim_{#1}}
\def\maximize#1{\displaystyle\maxim_{#1}}
\def\st{\mathop{\hbox{subject to}}}



\def\ip#1#2{\langle #1,#2\rangle}
\def\dps\displaystyle
\def\clink#1{\mkern2mu\overline{\mkern-2mu c\mkern2mu}\mkern-2mu\k(#1)}

\newcommand{\Real}{\mathbb{R}}

\newcommand{\Complex}{\mathbb{C}}

\def\abs#1{|#1|}
\def\norm#1{\|#1\|}

\newcommand{\indicator}[2]{\delta({#2}\mid{#1})}

\DeclareMathOperator{\diag}{diag}
\DeclareMathOperator{\Diag}{Diag}

\DeclareMathOperator{\trace}{trace}
\RedeclareMathOperator{\vec}{vec}

\def\half{{\textstyle{\frac{1}{2}}}}

\def\inv{^{-1}}

\def\spose#1{\hbox to 0pt{#1\hss}}
\def\text #1{\hbox{\quad#1\quad}}
\def\textt#1{\hbox{\qquad#1\qquad}}

\def\pthinsp{\mskip  2   mu}    %

\def\F{_{\scriptscriptstyle F}}

\def\L{_{\scriptscriptstyle L}}

\def\k{_k}

\def\km#1{_{k-#1}}

\def\lamhat{\skew{2.8}\widehat \lambda}
\def\lambdahat{\lamhat}

\def\bhat{\skew2\widehat b}

\def\hbar{\skew{4.2}\bar h}

\def\Uhat{\widehat U}

\def\Vhat{\widehat V}

\def\xhat{\skew{2.8}\widehat x}

\def\Xhat{\widehat X}

\def\yhat{\skew2\widehat y}

\def\Zhat{\widehat Z}

\def\xstarkm1{x^*\km1}

\def\ystarkm1{y^*\km1}

\newcommand{\Ascr}{\mathcal{A}}

\newcommand{\Cscr}{\mathcal{C}}

\newcommand{\Hscr}{\mathcal{H}}

\newcommand{\Kscr}{\mathcal{K}}

\newcommand{\Mscr}{\mathcal{M}}

\newcommand{\Oscr}{\mathcal{O}}
\newcommand{\Pscr}{\mathcal{P}}

\newcommand{\Xscr}{\mathcal{X}}
\newcommand{\Yscr}{\mathcal{Y}}

\def\Matlab{\textsc{Matlab}}

\usepackage{pgfplots}
\usepackage{pgfplotstable}
\usepackage{booktabs}
\usepackage{braket}
\usepackage{nicefrac}
\usepackage{algorithm2e}

\newcommand{\conj}[1]{\overline{#1}}
\newcommand{\conjtran}[1]{\overline{#1}^{\raisebox{-2.5pt}{$\pthinsp\scriptstyle T$}}}
\newcommand{\realpart}{\mathfrak{R}}
\newcommand{\imagpart}{\mathfrak{I}}
\newcommand{\Hermitian}{\Hscr}

\SetCommentSty{mycommfont}
\SetKwComment{Comment}{[\,}{\,]}

\title{LOW-RANK SPECTRAL OPTIMIZATION VIA GAUGE DUALITY%
  \thanks{August 1, 2015; revised August 12, 2015; revised February
    24, 2016; revised March 21, 2016.
}}

 \author{ Michael P. Friedlander%
  \thanks{%
    Department of Mathematics, University of California, Davis
    (\mailto{mpf@math.ucdavis.edu}). Research supported by ONR award
    N00014-16-1-2242.} \and Ives Mac\^edo%
  \thanks{%
    Department of Computer Science, University of British Columbia,
    Vancouver, BC, Canada.  (\mailto{ijamj@cs.ubc.ca}). Research
    supported by NSERC Discovery grant 312104 and NSERC Collaborative
    Research and Development grant 375142-08.} }
\date{\today}

\begin{document}

\maketitle
\thispagestyle{plain}
\pagestyle{myheadings}

\begin{abstract}
  Various applications in signal processing and machine learning give
  rise to highly structured spectral optimization problems
  characterized by low-rank solutions. Two important examples that
  motivate this work are optimization problems from phase retrieval
  and from blind deconvolution, which are designed to yield rank-1
  solutions. An algorithm is described that is based on solving a
  certain constrained eigenvalue optimization problem that corresponds
  to the gauge dual which, unlike the more typical Lagrange dual, has
  an especially simple constraint. The dominant cost at each iteration
  is the computation of rightmost eigenpairs of a Hermitian
  operator. A range of numerical examples illustrate the scalability
  of the approach.
\end{abstract}

\section{Introduction} \label{sec:intro}

There are a number of applications in signal processing and machine
learning that give rise to highly structured spectral optimization
problems. We are particularly interested in the class of problems
characterized by having solutions that are very low rank, and by
involving linear operators that are best treated by matrix-free
approaches. This class of problems is sufficiently restrictive that it
allows us to design specialized algorithms that scale well and lend
themselves to practical applications, but it is still sufficiently
rich to include interesting problems. Two examples include the
nuclear-norm minimization for problems such as blind deconvolution
\citep{Ahmed:2014}, and the PhaseLift formulation of the celebrated
phase retrieval problem \citep{candes2012phaselift}. The problems can
be cast generically in the semi-definite programming (SDP) framework,
for which a variety of algorithms are available. However, typical
applications can give rise to enormous optimization problems that
challenge the very best workhorse algorithms.

Denote the set of complex-valued $n\times n$ Hermitian matrices by $\Hscr^n$.
The algorithm that we propose is designed to solve the problems
\begin{subequations} \label{eq:primal-probs}
\begin{alignat}{3}
  \label{eq:sdp-primal}
  &\minimize{X\in\Hscr^n}
  &\quad &\trace X
  &\quad &\st\quad \norm{b-\Ascr X}\le\epsilon,\ X\succeq0,
\\\label{eq:nuc-primal}
  &\minimize{X\in\Complex^{n_1\times n_2}}
  &\quad &\norm{X}_1:=\sum_{i}\sigma_{i}(X)
  &\quad &\st\quad \norm{b-\Ascr X}\le\epsilon,
\end{alignat}
\end{subequations}
where the parameter $\epsilon$ controls the admissible deviations
between the linear model $\Ascr X$ and the vector of observations $b$.
(The particular properties of the vectors $b$ and of the linear
operators $\Ascr$ are detailed in section~\ref{sec:formulations}.)
Our approach for both problems is based on first solving a related
Hermitian eigenvalue optimization problem over a very simple
constraint, and then using that solution to recover a solution of the
original problem. This eigenvalue problem is highly structured, and
because the constraint is easily handled, we are free to apply a
projected first-order method with inexpensive per-iteration costs that
scales well to very large problems.

The key to the approach is to recognize that the
problems~\eqref{eq:primal-probs} are members of the family of gauge
optimization problems, which admit a duality concept different from
the Lagrange duality prevalent in convex optimization. Gauges are
nonnegative, positively homogeneous convex functions that vanish at
the origin. They significantly generalize the familiar notion of a
norm, which is a symmetric gauge function. The class of gauge
optimization problems, as defined by \citeauthor{freund:1987}'s
seminal 1987 work, can be stated simply: find the element of a convex
set that is minimal with respect to a gauge. These conceptually simple
problems appear in a remarkable array of applications, and include an
important cross-section of convex optimization. For example, all of
conic optimization can be phrased within the class of gauge
optimization; see \citet[Example 1.3]{FriedlanderMacedoPong:2014}, and
section~\ref{sec:gauge-duality} below.

Problem~\eqref{eq:sdp-primal} is not explicitly stated as a gauge
problem because the objective is not nonnegative everywhere on its
domain, as required in order for it to be a gauge function. It is,
however, nonnegative on the feasible set, and the problem can easily
be cast in the gauge framework simply by changing the objective
function to
\begin{equation}
  \label{eq:gauge-obj}
  \trace X + \indicator{\cdot\succeq0}{X},
  \text{where}
  \indicator{\cdot\succeq0}{X} =
  \begin{cases}
    0 & \mbox{if $X\succeq0$,}
    \\+\infty & \mbox{otherwise.}
  \end{cases}
\end{equation}
This substitution yields an equivalent problem, and the resulting
convex function is nonnegative and positively homogeneous---and
therefore a gauge function.  More generally, it is evident that any
function of the form $\gamma+\indicator{\Kscr}{\cdot}$ is a gauge, in
which $\gamma$ is a gauge and $\indicator{\Kscr}{\cdot}$ is the
indicator of a convex cone $\Kscr$.

The method that we develop applies to the much broader class of
semidefinite optimization problems with nonnegative objective values,
as described in section~\ref{sec:extensions}.  We pay special
attention to the low-rank spectral problems just mentioned because
they have a special structure that can be exploited both theoretically
and computationally.

\subsection{Notation} \label{sec:notation}

To emphasize the role of the vector of singular values $\sigma(X),$ we
adopt the Schatten $p$-norm notation for the matrix-norms referenced
in this paper, i.e., $\norm{X}_p:=\norm{\sigma(X)}_p$. Thus, the
nuclear, Frobenius, and spectral norms of a matrix $X$ are denoted by
$\norm{X}_1,$ $\norm{X}_2$, and $\norm{X}_\infty,$ respectively.  The
notation for complex-valued quantities, particularly in the SDP
context, is not entirely standard. Here we define some objects we use
frequently.  Define the complex inner product
$\ip{X}{Y}:=\trace XY^*$, where $Y^{*}$ is the conjugate transpose of
a complex matrix $Y$, i.e., $Y^{*}=\conjtran{Y}$.  The set of
$n\times n$ Hermitian matrices is denoted by $\Hermitian^{n}$, and
$X\succeq0$ (resp., $X\succ0)$ indicates that the matrix $X$ is both
Hermitian and positive semidefinite (resp., definite). Let
$\lambda(A)$ be the vector of ordered eigenvalues of
$A\in\Hermitian^{n}$, i.e.,
$\lambda_{1}(A)\ge\lambda_{2}(A)\ge\cdots\ge\lambda_{n}(A).$ (An
analogous ordering is assumed for the vector of singular values.) For
$B\succ0,$ let $\lambda(A,B)$ denote the vector of generalized
eigenvalues of the pencil $(A,B)$, i.e.,
$\lambda(A,B)=\lambda(B^{-\nicefrac12}AB^{-\nicefrac12}).$ \linebreak
Let $\realpart(\cdot)$ and $\imagpart(\cdot)$ denote the real and
imaginary parts of their arguments. The norm dual to
$\norm{\cdot}:\Complex^m\to\Real_{+}$ is defined by
\begin{equation}\label{eq:5}
\norm{x}_{*}:=\sup_{\norm{z}\le1}\,\realpart\ip{z}{x}.
\end{equation}
The positive part of a scalar is denoted by $[\cdot]_{+}=\max\{0,\cdot\}.$

When we make reference to \emph{one-} and \emph{two-dimensional}
signals, our intent is to differentiate between problems that involve
discretized functions of \emph{one} and \emph{two} variables,
respectively, rather than to describe the dimension of the ambient
space. Hence the terms \emph{two-dimensional signals} and
\emph{two-dimensional images} are used interchangeably.

Generally we assume that the problems are feasible, although we
highlight in section~\ref{sec:gauge-duality} how to detect infeasible
problems. We also assume that $0\le\epsilon<\norm{b}$, which ensures
that the origin is not a trivial solution. In practice, the choice of
the norm that defines the feasible set will greatly influence the
computational difficulty of the problem. Our implementation is based
on the 2-norm, which often appears in many practical applications.
Our theoretical developments, however, allow for any norm.

\subsection{Problem formulations} \label{sec:formulations}

Below we describe two applications, in phase retrieval and blind
deconvolution, that motivate our work.  There are other relevant
examples, such as matrix completion \citep{RFP:2010}, but these two
applications require optimization problems that exemplify properties
we exploit in our approach.

\subsubsection{Phase retrieval} \label{ssec:pr}

The phase retrieval problem is concerned with recovery of the phase
information of a signal---e.g., an image---from magnitude-only
measurements.  One important application is X-ray crystallography,
which generates images of the molecular structure of  crystals
\citep{Harrison:93}. Other applications are described by
\cite{candes2012phaselift} and \cite{Waldspurger:2015}. They describe
the following recovery approach, based on convex optimization.

Magnitude-only measurements of the signal $x\in\Complex^n$ can be
described as quadratic measurements of the form
\[
 b_{k} = \abs{\ip{x}{a_k}}^2
\]
for some vectors $a_k$ that encode the waveforms used to illuminate
the signal. These quadratic measurements of $x$ can be understood as
linear measurements
\[
  b\k = \ip{xx^*}{a_k a_k^*}
      = \ip{X}{A_k}
\]
of the lifted signal $X:=xx^*$, where $A_k:=a_ka_k^*$ is the
$k$th lifted rank-1 measurement matrix.

In the matrix space, the trace of the unknown lifted signal $X$ acts
as a surrogate for the rank function. This is analogous to the 1-norm,
which stands as a convex surrogate for counting the number of nonzeros
in a vector. This leads us to an optimization problem of the
form~\eqref{eq:sdp-primal}, where $\Ascr:\Hermitian^{n}\to\Real^{m}$
is defined by $(\Ascr X)_{k}:=\ip{X}{A\k}$.  The parameter $\epsilon$
anticipates noise in the measurements. \citet{candes2012phaselift}
call this the \emph{PhaseLift} formulation. In section~\ref{ssec:plexp} we
give numerical examples for recovering one- and two-dimensional
signals with and without noise.

\subsubsection{Biconvex compressed sensing and blind
  deconvolution} \label{ssec:bcs}

The biconvex compressed sensing problem \citep{Ling:2015} aims to
recover two signals from a number of sesquilinear measurements of the
form
\[
  b_k = \ip{x_1}{a_{1k}}\conj{\ip{x_2}{a_{2k}}},
\]
where $x_{1}\in\Complex^{n_{1}}$ and $x_{2}\in\Complex^{n_{2}}$. In
the context of blind deconvolution, $x_{1}$ and $x_{2}$ correspond to
coefficients of the signals in some bases.  The lifting approached
used in the phase retrieval formulation can again be used, and the
measurements of $x_{1}$ and $x_{2}$ can be understood as coming from
linear measurements
\[
  b\k = \ip{x_{1}x_{2}^{*}}{a_{1k}a_{2k}^{*}} = \ip{X}{A_k}
\]
of the lifted signal $X=x_1x_2^*$, where $A_k:=a_{1k}a_{2k}^*$ is the
lifted asymmetric rank-1 measurement matrix. \citet{Ahmed:2014} study
conditions on the structure and the number of measurements that
guarantee that the original vectors (up to a phase) may be recovered
by minimizing the sum of singular values of $X$ subject to the linear
measurements. This leads to an optimization problem of the
form~\eqref{eq:nuc-primal}, where
$\Ascr:\Complex^{n_1\times n_2}\to\Complex^m$ is defined by
$(\Ascr X)_k:=\ip{X}{A_k}.$

In section~\ref{ssec:bcsexp}, we describe a two-dimensional blind
deconvolution application from motion deblurring, and there we provide
further details on the structure of the measurement operators $A\k$,
and report on numerical experiments.

\subsection{Reduction to Hermitian SDP} \label{sec:gform}

It is convenient, for both our theoretical and algorithmic development,
to embed the nuclear-norm minimization problem~\eqref{eq:nuc-primal}
within the symmetric SDP~\eqref{eq:sdp-primal}. The resulting theory
is no less general, and it permits us to solve both problems with what
is essentially a single software implementation.  The reduction to the
Hermitian trace-minimization problem~\eqref{eq:sdp-primal} takes the
form
\begin{equation} \label{eq:11}
\begin{aligned}
  &\minimize{\substack{U\in\Hscr^{n_1},V\in\Hscr^{n_2}
             \\\quad X\in\Complex^{n_1\times n_2}\\r_{1},r_{2}\in\Real^{m}}}
  & &\quad \frac12\left\langle\begin{pmatrix}I&0\\0&I\end{pmatrix},
    \begin{pmatrix}U&X\\X^*&V\end{pmatrix}\right\rangle
  \\&\;\;\st & &\quad
       \begin{aligned}[t]
         \frac{1}{2}\left\langle
           \begin{pmatrix}0&A_k\\\phantom{-}A_k^*&0\end{pmatrix},
           \begin{pmatrix}U&X\\X^*&V\end{pmatrix}
         \right\rangle
         +r_{1k}
         &= \realpart b_k,
       \\\frac{i}{2}\left\langle
         \begin{pmatrix}0&A_k\\-A_k^*&0\end{pmatrix},
         \begin{pmatrix}U&X\\X^*&V\end{pmatrix}
       \right\rangle
       +r_{2k}
       &= \imagpart b_k,
      \\\norm{r_{1}+i r_{2}}\le\epsilon,
      \quad
      \begin{pmatrix}U&X\\X^*&V\end{pmatrix}&\succeq0,\ \ k=1,\ldots,m.
    \end{aligned}
\end{aligned}
\end{equation}

The residual variables $r_{1},r_{2}\in\Real^m$ merely serve to allow
the compact presentation above, as they can be eliminated using the
equality constraints. The additional variables $U$ and $V,$ at the
minimizer, correspond to $(XX^*)^{\nicefrac12}$ and
$(X^*X)^{\nicefrac12},$ respectively; the variable $X$ retains its
original meaning.  This reduction is based on a well-known
reformulation of the nuclear norm as the optimal value of an SDP;
see~\citet[Lemma 2]{Fazel:2002} or~\citet[Proof of Proposition
2.1]{RFP:2010}.

Although this reduction is convenient for our presentation, it is not
strictly necessary, as will be clear from the results
in section~\ref{sec:extensions}.  In fact, the resulting increase in problem
size might affect a solver's performance, as would likely be
noticeable if dense solvers are employed.  Our focus, however, is on
large problems that require matrix-free operators and exhibit low-rank
solutions. Throughout this paper, we focus entirely on the SDP
formulation~\eqref{eq:sdp-primal} without loss of generality.

\subsection{Approach} \label{ssec:approach}

Our strategy for these low-rank spectral optimization problems is
based on solving the constrained eigenvalue optimization problem
\begin{equation}
  \label{eq:gaugedual}
  \minimize{y\in\Real^m} \quad \lambda_1(\Ascr^*y)
  \quad\st\quad
    \ip b y - \epsilon\norm{y}_*\ge1
\end{equation}
that results from applying gauge duality
\citep{FriedlanderMacedoPong:2014,freund:1987} to a suitable
reformulation of~\eqref{eq:sdp-primal}. This is outlined in
section~\ref{sec:gauge-duality}.  The dimension of the variable $y$ in the
eigenvalue optimization problem corresponds to the number of
measurements. In the context of phase retrieval and blind
deconvolution, \cite{Candes:2014:SQE:2673201.2673259} and
\cite{Ahmed:2014} show that the number of measurements needed to
recover with high probability the underlying signals is within a
logarithmic factor of the signal length.  The crucial implication is
that the dimension of the dual problem grows slowly as compared to the
dimension of the primal problem, which grows as the square of the
signal length.

In our implementation, we apply a simple first-order projected
subgradient method to solve the eigenvalue problem. The dominant cost
at each iteration of our algorithm is the computation of rightmost
eigenpairs of the $n\times n$ Hermitian linear operator $\Ascr^{*}y$,
which are used to construct descent directions
for~\eqref{eq:gaugedual}. The structure of the measurement operators
allows us to use Krylov-based eigensolvers, such as ARPACK
\citep{lehoucq1998arpack}, for obtaining these leading eigenpairs.
Primal solution estimates $X$ are recovered via a relatively small
constrained least-squares problem, described in
section~\ref{sec:implementation}.

An analogous approach based on the classical Lagrangian duality 
also leads to a dual optimization problem in the same space as our dual
eigenvalue problem:
\begin{equation}
  \label{eq:sdpld}
  \maximize{y\in\Real^m}\quad \ip b y - \epsilon\norm{y}_*
  \quad\st\quad
  \Ascr^*y\preceq I.
\end{equation}
Note that the Lagrange dual possesses a rather simple objective and a
difficult linear matrix inequality of order $n$ as a constraint. Precisely the
reverse situation holds for the gauge dual~\eqref{eq:gaugedual}, which
has a relatively simple constraint. %

It is well known that SDPs with a constant-trace property---i.e.,
$\Ascr X=b$ implies $\trace(X)$ is constant---have Lagrange dual
problems that can be formulated as unconstrained eigenvalue
problems. This approach is used by \cite{HelmbergRendl:2000} to
develop a spectral bundle method. The applications that we consider,
however, do not necessarily have this property.

\subsection{Reproducible research} The data files and \Matlab\ scripts
used to generate the numerical results presented in
section~\ref{sec:experiments} can be obtained at the following URL:
\begin{center}
  \smallskip
  \url{http://www.cs.ubc.ca/~mpf/low-rank-opt}
\end{center}

\subsection{Related work} Other researchers have recognized the need
for algorithms with low per-iteration costs that scale well for
large-scale, low-rank spectral optimization problems. Notable efforts
include \cite{hazan2008sparse}, \cite{laue2012hybrid}, and
\cite{2015arXiv151102204F}, who advocate variations of the Frank-Wolfe
(FW) method to solve some version of the problem
\begin{equation} \label{eq:frank-wolfe-prob}
 \minimize{X} \quad f(X) \quad\st\quad \trace(X)\le\tau,\ X\succeq0,
\end{equation}
where $f$ is a differentiable function. For example, the choice
$f(X)=\half\norm{\Ascr X-b}_2^2$ yields a problem related
to~\eqref{eq:sdp-primal}. The asymmetric version of the problem with
$\norm{X}_1\le\tau$ is easily accommodated by simply replacing the
above constraints. For simplicity, here we focus on the symmetric
case, though our approach applies equally to the asymmetric case.  The
main benefit of using FW for this problem is that each iteration
requires only a rightmost eigenvalue of the gradient $\nabla f(X)$,
and therefore has the same per-iteration cost of the method that we
consider, which requires a rightmost eigenvalue of the same-sized
matrix $\Ascr^* y$. The same Krylov-based eigensolvers apply in both
cases.

There are at least two issues that need to be addressed when comparing
the FW algorithm to the approach we take here. First, as
\cite{2015arXiv151102204F} make clear, even in cases where low-rank
solutions are expected, it is not possible to anticipate the rank of
early iterates $X_k$ generated by the FW method. In particular, they
observe that the rank of $X_k$ quickly increases during early
iterations, and only slowly starts to decrease as the solution is
approached. This motivates their development of algorithmic devices
that attenuate rank growth in intermediate iterates.  Any
implementation, however, must be prepared to increase storage for the
factors of $X_k$ during intermediate iterates. In contrast, a
subgradient method applied to the gauge dual problem can be
implemented with constant storage. Second, although in principle there
exists a parameter $\tau$ that causes the optimization
problems~\eqref{eq:frank-wolfe-prob} and~\eqref{eq:sdp-primal} to
share the same solution, this parameter is not generally known in
advance. One way around this is to solve a sequence of
problems~\eqref{eq:frank-wolfe-prob} for varying parameters $\tau\k$
using, for example, a level-set procedure described by
\cite*{2016arXiv160201506A}.

As an alternative to applying the FW algorithm
to~\eqref{eq:frank-wolfe-prob}, we might instead consider applying a
variation of the FW method directly to the gauge dual
problem~\eqref{eq:gaugedual}. Because the gauge dual objective is not
differentiable and the feasible set is not compact if $\epsilon=0$,
some modification to the standard FW method is required.
\cite{argyriou2014hybrid}, \cite{bach2015duality}, and
\cite{nesterov2015complexity} propose variations of FW that involve
smoothing the objective. These smoothing approaches are typically
based on infimal convolution with a smooth kernel, which may lead to a
function whose gradient is expensive to compute. For example, the
``soft-max'' smooth approximation of $\lambda_1(\cdot)$ is the
function
$\mu\log\sum_{i=1,\ldots,n}\exp(\lambda_i(\cdot)/\mu)$. Forming the
gradient of this smooth function requires computing all
eigenvalues of an $n$-by-$n$ Hermitian matrix.

\citet{laue2012hybrid} proposes a hybrid algorithm that interleaves a
nonconvex subproblem within the FW iterations. If the local minimum of
the nonconvex subproblem improves the objective value, it is used to
replace the current FW iterate. This approach is similar in spirit to
the primal-dual refinement that we describe in
section~\ref{sec:pr-du-refinement}, but because Laue's method is entirely
primal, it has the benefit of not requiring a procedure to feed the
improved primal sequence back to the dual sequence.
 
\section{Spectral gauge optimization and duality} \label{sec:gauge-duality}

The derivation of the eigenvalue optimization
problem~\eqref{eq:gaugedual} as a dual to~\eqref{eq:sdp-primal}
follows from a more general theory of duality for gauge
optimization. Here we provide some minimal background for our
derivations related to spectral optimization; see \cite{freund:1987}
and \cite{FriedlanderMacedoPong:2014} for fuller descriptions. We
begin with a general description of the problem class.

Let $\kappa:\Xscr\mapsto\Real\cup\{+\infty\}$ and
$\rho:\Yscr\mapsto\Real\cup\{+\infty\}$ be gauge functions, where
$A:\Xscr\mapsto\Yscr$ is a linear operator that maps between the
finite-dimensional real inner-product spaces $\Xscr$ and $\Yscr$. The
polar
\[
 f^{\circ}(y) := \inf\set{\mu>0 | \ip{x}{y} \le \mu f(x)\ \forall x}
\]
of a gauge $f$ plays a key role in the duality of gauge problems. The
 problems
\begin{subequations} \label{eq:12}
\begin{alignat}{4}
  \label{eq:gp}
  &\minimize{x\in\Xscr} &\quad &\kappa(x)
  & &\quad\st\quad &  \rho(b-Ax)&\le\epsilon,
\\\label{eq:gd}
  &\minimize{y\in\Yscr} &\quad &\kappa^\circ(A^*y)
  & &\quad\st\quad &  \ip b y - \epsilon\rho^\circ(y)&\ge1,
\end{alignat}
\end{subequations}
are dual to each other in the following sense: all primal-dual
feasible pairs $(x,y)$ satisfy the weak-duality relationship
\begin{equation}
  \label{eq:1}
  1\le \kappa(x)\,\kappa^{\circ}(A^{*}y).
\end{equation}
Moreover, a primal-dual feasible pair is optimal if this holds with
equality. This strong-duality relationship provides a certificate of
optimality.

The SDP problem \eqref{eq:sdp-primal} can be cast into the mold of the
canonical gauge formulation~\eqref{eq:gp} by using the redefined
objective~\eqref{eq:gauge-obj} and making the identifications
\begin{equation*}
  \kappa(X)=\trace X+\delta(X\,|\,\cdot\succeq0)
  \textt{and}
  \rho(r) = \norm{r}.
\end{equation*}
We use the polar calculus described by \citet[Section
7.2.1]{FriedlanderMacedoPong:2014} together with the definition of the dual norm
to obtain the correponding polar functions:
\begin{equation*}
  \kappa^\circ(Y)=[\lambda_1(Y)]_+
  \textt{and}
  \rho^{\circ}(y)=\norm{y}_{*}.
\end{equation*}
It then follows from~\eqref{eq:12} that the following are a dual gauge pair:
\begin{subequations}\label{eq:13}
\begin{alignat}{3}
  \label{eq:sdpgp}
  &\minimize{X\in\Hscr^n} \quad \trace X + \delta(X\,|\,\cdot\succeq0)
  &\quad&\st&\quad
  \norm{b-\Ascr X}&\le\epsilon,
  \\
  \label{eq:3}
  &\minimize{y\in\Real^m}\quad\qquad [\lambda_1(\Ascr^*y)]_{+}
  &\quad&\st&\quad
  \ip b y - \epsilon\norm{y}_*&\ge1.
\end{alignat}
\end{subequations}
The derivation of gauge dual problems relies on the polarity
operation applied to gauges. When applied to a norm, for example, the
polar is simply the dual norm. In contrast, Lagrange duality is
intimately tied to conjugacy, which is what gives rise to the dual
problem \eqref{eq:sdpld}. Of course, the two operations are closely
related. For any guage function $\kappa$, for example,
$\kappa^*(y)=\delta_{\kappa^\circ(\cdot)\le1}(y)$. These relationships
are described in detail by \citet[Section~15]{Roc70} and
\citet[Section 2.3]{FriedlanderMacedoPong:2014}.

We can simplify the dual objective and safely eliminate the
positive-part operator: because $\kappa(X)$ is necessarily strictly
positive for all nonzero $X$, and is additionally finite over the
feasible set of the original problem~\eqref{eq:sdp-primal}, it follows
from~\eqref{eq:1} that $\kappa^{\circ}(\Ascr^{*}y)$ is positive for
all dual feasible points. In other words,
\begin{equation}\label{eq:2}
 0 < [\lambda_{1}(\Ascr^{*}y)]_{+} = \lambda_{1}(\Ascr^{*}y)
\end{equation}
for all dual feasible points $y$. Hence we obtain the equivalent dual
problem~\eqref{eq:gaugedual}.

In practice, we need to be prepared to detect whether the primal
problem~\eqref{eq:sdpgp} is infeasible. The failure of
condition~\eqref{eq:2} in fact furnishes a certificate of
infeasibility for~\eqref{eq:sdp-primal}: if $\lambda(\Ascr^{*}y)=0$
for some dual-feasible vector $y$, it follows from~\eqref{eq:1} that
$\kappa(X)$ is necessarily infinite over the feasible set
of~\eqref{eq:sdpgp}---i.e., $X\not\succeq0$ for all $X$ feasible
for~\eqref{eq:sdpgp}. Thus,~\eqref{eq:sdp-primal} is infeasible.

\section{Derivation of the approach} \label{sec:gauge-sdp}

There are two key theoretical pieces needed for our approach. The
first is the derivation of the eigenvalue optimization
problem~\eqref{eq:gaugedual}, as shown in section~\ref{sec:gauge-duality}.
The second piece is the derivation of a subproblem that allows recovery
of a primal solution $X$ from a solution of the eigenvalue
problem~\eqref{eq:gaugedual}.

\subsection{Recovering a primal solution} \label{ssec:sd}

Our derivation of a subproblem for primal recovery proceeds in two
stages. The first stage develops necessary and sufficient optimality
conditions for the primal-dual gauge pair~\eqref{eq:sdpgp}
and~\eqref{eq:3}. The second stage uses these to derive a subproblem
that can be used to recover a primal solution from a dual solution.

\subsubsection{Strong duality and optimality conditions}

The weak duality condition~\eqref{eq:1} holds for all primal-feasible
pairs $(X, y)$. The following result asserts that if the pair is
optimal, then that inequality must necessarily hold tightly.

\begin{proposition}[Strong duality]
  \label{prop:strong-duality}
  If~\eqref{eq:sdp-primal} is feasible and $0\le\epsilon<\norm{b},$
  then
  \begin{equation}\label{eq:4}
    \left[
      \min_{
        \substack{X\in\Hscr^n\\\norm{b-\Ascr X}\leq\epsilon}}
      \trace X+\delta(X\,|\,\cdot\succeq0)\rule{0pt}{25pt}
    \right]
    \cdot
    \left[
      \inf_{\substack{y\in\Real^m\\\ip b y -
          \epsilon\norm{y}_*\geq1}}[\lambda_1(\Ascr^*y)]_+
    \right]
    =1.
  \end{equation}
\end{proposition}
\begin{proof}
  We proceed by reasoning about the Lagrangian-dual
  pair~\eqref{eq:sdp-primal} and~\eqref{eq:sdpld}. We then
  translate these results to the corresponding gauge-dual
  pair~\eqref{eq:sdpgp} and~\eqref{eq:3}.

  The primal problem~\eqref{eq:sdp-primal} is feasible by assumption.
  Because its Lagrange dual problem~\eqref{eq:sdpld} admits strictly
  feasible points (e.g., $y=0$), it follows from \citet[Theorems~28.2
  and~28.4]{Roc70} that the primal problem attains its positive
  minimum value and that there is zero duality gap between the
  Lagrange-dual pair.

  Moreover, because the primal problem \eqref{eq:sdp-primal} attains
  its positive minimum value for some $\Xhat$, and there is zero
  duality gap, there exists a sequence $\{y_j\}$ such that
  $[\lambda_1(\Ascr^*y_j)]_+\leq1$ and
  $\ip{y_j}{b}-\epsilon\norm{y_j}_*\nearrow\trace\Xhat.$ Because
  $\trace\Xhat>0$, we can take a subsequence $\{y_{j_k}\}$ for which
  $\ip{y_{j_k}}{b}-\epsilon\norm{y_{j_k}}_*$ is uniformly bounded
  above zero. Define the sequence $\{\yhat_k\}$ by
  $\yhat_k:=y_{j_k}(\ip{y_{j_k}}{b}-\epsilon\norm{y_{j_k}}_*)^{-1}$.
  Then $\ip{\yhat_k}{b}-\epsilon\norm{\yhat_k}_*=1$ for all $k$,
  which is a feasible sequence for the gauge dual
  problem~\eqref{eq:3}. Weak gauge duality~\eqref{eq:1} and the
  definition of $\yhat_k$ then implies that
  \[
    (\trace\Xhat)^{-1}\leq[\lambda_1(\Ascr^*\yhat_k)]_+\leq(\ip{y_{j_k}}{b}-\epsilon\norm{y_{j_k}}_*)^{-1}\searrow(\trace\Xhat)^{-1}.
  \]
  Multiply the series of inequalities by $\trace\Xhat$ to
  obtain~\eqref{eq:4}.
\end{proof}

Note the lack of symmetry in the statement of
Proposition~\ref{prop:strong-duality}: the primal problem is stated
with a ``min'', but the dual problem is stated with an ``inf''. This
is because the dual Slater condition---i.e., strict feasibility of the
corresponding Lagrange-dual problem~\eqref{eq:sdpld}---allows us to
assert that a primal optimal solution necessarily exists.  However, we
cannot assert in general that a dual optimal solution exists because
the corresponding primal feasible set does not necessarily satisfy the
Slater condition.

Although in this work we do not attempt to delineate conditions under
which dual attainment holds, a practical case in which it always does
is when the primal objective is a norm and the measurement operator is
surjective. In that case, the dual gauge objective
$\norm{\Ascr^*\cdot}_*$ defines a norm in $\Real^m$, which has
compact level sets. Hence a dual solution always exists.  We comment
further on this theoretical question in section~\ref{sec:conclusions}.

The following result characterizes gauge primal-dual optimal pairs. It
relies on von Neumann's trace inequality: for Hermitian matrices $A$
and $B$,
\[
  \ip{A}{B}\le\ip{\lambda(A)}{\lambda(B)},
\]
and equality holds if and only if $A$ and $B$ admit a simultaneous
ordered eigendecomposition, i.e., $A = U\Diag[\lambda(A)]U^{*}$ and
$B=U\Diag[\lambda(B)]U^{*}$ for some unitary matrix~$U$;
see~\cite{lewis1996convex}.

\begin{proposition}[Optimality conditions]
  \label{prop:optimality} If~\eqref{eq:sdp-primal} is feasible and
  $0\leq\epsilon<\norm{b},$ then $(X,y)\in\Hscr^n\times\Real^m$ is
  primal-dual optimal for the gauge dual pair \eqref{eq:sdpgp} and
  \eqref{eq:3} if and only if the following conditions hold:
  \begin{enumerate}
    \item\label{prop:sd:pf} $X\succeq0$ and $\norm{b-\Ascr X}=\epsilon;$
    \item\label{prop:sd:df} $\ip{y}{b}-\epsilon\norm{y}_*=1;$
    \item\label{prop:sd:cs} $\ip{y}{b-\Ascr X}=\norm{y}_*\norm{b-\Ascr X};$
    \item\label{prop:sd:cp} $\lambda_i(X)\cdot(\lambda_1(\Ascr^*y)-\lambda_i(\Ascr^*y))=0$, $i=1,\ldots,n;$
    \item\label{prop:sd:vn} $X$ and $\Ascr^*y$
      admit a simultaneous ordered eigendecomposition.
  \end{enumerate}
\end{proposition}
\begin{proof}
  By strong duality (Proposition~\ref{prop:strong-duality}), the pair
  $(X,y)\in\Hscr^m\times\Real^m$ is primal-dual optimal if and only if
  they are primal-dual feasible and the product of their corresponding
  objective values is equal to one. In this case,
  \begin{align*}
    1&=[\trace X+\delta(X\,|\,\cdot\succeq0)]
       \cdot[\lambda_1(\Ascr^*y)]_+\tag{strong duality}\\
     &=\ip{e}{\lambda(X)}\cdot\lambda_1(\Ascr^*y)\\
     &=\ip{\lambda_1(\Ascr^*y)\cdot e}{\lambda(X)}\\
     &\ge\ip{\lambda(\Ascr^*y)}{\lambda(X)}\tag{$\lambda_1(\Ascr^*y)\geq\lambda_i(\Ascr^*y)$ and $X\succeq0$}\\
     &\ge\ip{\Ascr^*y}{X}\tag{von Neumann's trace inequality}\\
     &=\ip{y}{\Ascr X}\\
     &=\ip{y}{b}-\ip{y}{b-\Ascr X}\\
     &\ge\ip{y}{b}-\norm{y}_*\norm{b-\Ascr X}
       \tag{Cauchy-Schwartz inequality}\\
     &\ge\ip{y}{b}-\epsilon\norm{y}_*
       \tag{primal feasibility}\\
     &\ge1.\tag{dual feasibility}
  \end{align*}
  Thus all of the above inequalities hold with equality. This proves
  conditions 1--4. Condition 5 follows from again invoking von
  Neumann's trace inequality and noting its implication that $X$ and
  $\Ascr^{*}y$ share a simultaneous ordered eigenvalue decomposition.
  Sufficiency of those conditions can be verified by simply following
  the reverse chain of reasoning and again noticing that the
  inequalities can be replaced by equalities.
\end{proof}

\subsection{Primal recovery subproblem} \label{ssec:pfd}

The optimality conditions stated in Proposition~\ref{prop:optimality}
furnish the means for deriving a subproblem that can be used to
recover a primal solution from a dual solution. The next result
establishes an explicit relationship between primal solutions $X$ and
$\Ascr^{*}y$ for an arbitrary optimal dual solution $y$.

\begin{corollary}\label{corl:pfd}
  Suppose that the conditions of Proposition~\ref{prop:optimality}
  hold. Let $y\in\Real^m$ be an arbitrary optimal solution for the
  dual gauge program~\eqref{eq:3}, $r_1\in\{1,\ldots,n\}$ be the
  multiplicity of $\lambda_1(\Ascr^*y)$, and
  $U_1\in\Complex^{n\times r_1}$ be the matrix formed by the first
  $r_1$ eigenvectors of $\Ascr^*y.$ Then a matrix $X\in\Hscr^n$ is a
  solution for the primal problem~\eqref{eq:sdpgp} if and only if
  there exists an $r_{1}\times r_{1}$ matrix $S\succeq0$ such that
  \begin{equation}\label{eq:pdf-opt}
  X=U_1SU_1^*
  \textt{and}
  (b-\Ascr X)\in\epsilon\partial\norm{\cdot}_*(y).
  \end{equation}

\end{corollary}
\begin{proof}
  The assumptions imply that the optimal dual value is positive.  If
  $y\in\Real^m$ is an optimal solution to~\eqref{eq:3}, the
  positive-homogeneity of its objective and constraint, and the
  positivity of the optimal value, allow us to deduce that the dual
  constraint must be active, i.e., $\ip{y}{b}-\epsilon\norm{y}_*=1.$
  Thus condition~\ref{prop:sd:df} of
  Proposition~\ref{prop:optimality} holds.

  The construction of $X$ in~\eqref{eq:pdf-opt} guarantees that it
  shares a simultaneous ordered eigendecomposition with $\Ascr^{*}y$,
  and that it has rank of $r_{1}$ at most. Thus,
  conditions~\ref{prop:sd:cp} and~\ref{prop:sd:vn} of
  Proposition~\ref{prop:optimality} hold.

  We now show that conditions~1 and~3 of the proposition hold. The
  subdifferential $\partial\norm{\cdot}_{*}$ corresponds to the set of
  maximizers of the linear function that defines the dual ball;
  see~\eqref{eq:5}. Then because
  $(b-\Ascr X)\in\epsilon\partial\norm{\cdot}_*(y),$ it holds that
  $\norm{b-\Ascr X}\leq\epsilon$ and
  $\epsilon\norm{y}_*=\ip{y}{b-\Ascr X}\leq\norm{y}_*\norm{b-\Ascr
    X}\leq\epsilon\norm{y}_*,$
  implying that $\norm{b-\Ascr X}=\epsilon$ and
  $\ip{y}{b-\Ascr X}=\norm{y}_*\norm{b-\Ascr X}.$ This way,
  condition~\ref{prop:sd:pf} and~\ref{prop:sd:cs} of
  Proposition~\ref{prop:optimality} are also satisfied. Hence, all
  the conditions of the proposition are satisfied, and the pair
  $(X,y)\in\Hscr^n\times\Real^m$ is optimal.

  Suppose now that $X\in\Hscr^n$ is optimal for~\eqref{eq:sdpgp}. We
  can invoke Proposition~\ref{prop:optimality} on the pair
  $(X,y)\in\Hscr^n\times\Real^m.$ Condition~\ref{prop:sd:cp} implies that
  any eigenvector of $\Ascr^*y$ associated to an eigenvalue
  $\lambda_i(\Ascr^*y)$ with $i>r_1$ is in the nullspace of $X,$
  therefore there is an $r_{1}\times r_{1}$ matrix $S\succeq0$ such
  that $X=U_1SU_1^*.$ Conditions~\ref{prop:sd:pf} and~\ref{prop:sd:cs}
  imply that $\norm{b-\Ascr X}\leq\epsilon$ and
  $\ip{y}{b-\Ascr X}=\epsilon\norm{y}_*,$ thus verifying that
  $(b-\Ascr X)\in\epsilon\partial\norm{\cdot}_*(y)$, as required.
\end{proof}

Corollary~\ref{corl:pfd} thus provides us with a way to recover a
solution to our model problem~\eqref{eq:sdp-primal} after computing a
solution to the gauge dual problem~\eqref{eq:sdpgp}.  When the
residual in~\eqref{eq:sdp-primal} is measured in the 2-norm,
condition~\eqref{eq:pdf-opt} simplifies, and implies that the matrix
$S$ that defines $X=USU^{*}$ can be obtained by solving
\begin{equation} \label{eq:primal-recovery}
  \minimize{S\succeq0}\quad\norm{\Ascr(U_{1}SU_{1}^{*})-b_{\epsilon}}^{2},
  \text{with}
  b_{\epsilon}:=b-\epsilon y/\norm{y}.
\end{equation}
When the multiplicity $r_1$ of the eigenvalue $\lambda_1(\Ascr^*y)$ is
much smaller than $n$, this optimization problem is relatively
inexpensive. In particular, if $r_{1}=1$---which may be expected in some
applications such as PhaseLift---the optimization problem is over a
scalar $s$ that can be obtained immediately as
\begin{equation*}
 s = [\ip{\Ascr(u_{1}u_{1}^*)}{b_\epsilon}]_{+}/\norm{\Ascr(u_{1}u_{1}^*)}^{2}
\end{equation*}
where $u_{1}$ is the rightmost eigenvalue of $\Ascr^{*}y$.  This
approach exploits the complementarity relation on eigenvalues in
condition~\ref{prop:sd:cp} of Proposition~\ref{prop:optimality} to
reduce the dimensionality of the primal solution recovery. Its
computational difficulty effectively depends on finding a dual
solution $y$ at which the rightmost eigenvalue has low multiplicity
$r_1$.

\section{Implementation} \label{sec:implementation}

The success of our approach hinges on efficiently solving the
constrained eigenvalue optimization problem~\eqref{eq:sdpld} in order
to generate solution estimates $y$ and rightmost eigenvector estimates
$U_{1}$ of $\Ascr^{*}y$ that we can feed
to~\eqref{eq:primal-recovery}. The two main properties of this problem
that drive our approach are that it has a nonsmooth objective and that
projections on the feasible set are inexpensive. Our implementation is
based on a basic projected-subgradient descent method, although
certainly other choices are available. For example,
\cite{nesterov2009unconstrained} and \cite{richtarik2011improved}
propose specialized algorithms for minimizing positively homogeneous
functions with affine constraints; some modification of this approach
could possibly apply to~\eqref{eq:sdpld}. Another possible choice is
Helmberg and Rendl's (\citeyear{HelmbergRendl:2000}) spectral bundle
method. For simplicity, and because it has proven sufficient for our
needs, we use a standard projected subgradient method, described
below.

\subsection{Dual descent} \label{sec:dual-descent}

The generic subgradient method is based on the iteration
\begin{equation}\label{eq:6}
  y_{+} = \Pscr(y - \alpha g),
\end{equation}
where $g$ is a subgradient of the objective at the current iterate
$y$, $\alpha$ is a positive steplength, and the operator
$\Pscr:\Real^{m}\to\Real^{m}$ gives the Euclidean projection onto the
feasible set. For the objective function
$f(y) = \lambda_{1}(\Ascr^{*}y)$ of~\eqref{eq:gaugedual}, the
subdifferential has the form
\begin{align} \label{eq:7}
\partial f(y)
  = \set{ \Ascr(U_{1}TU_{1}^{*}) | T\succeq0,\ \trace T=1 },
\end{align}
where $U_{1}$ is the $n\times r_{1}$ matrix of rightmost eigenvectors
of $\Ascr^{*}y$ \citep[Theorem~3]{doi:10.1137/0802007}. A Krylov-based
eigenvalue solver can be used to evaluate $f(y)$ and a subgradient
$g\in\partial f(y)$. Such methods require products of the form
$(\Ascr^{*}y)v$ for arbitrary vectors~$v$. In many cases, these
products can be computed without explicitly forming the matrix
$\Ascr^{*}y$. In particular, for the applications described
in section~\ref{sec:formulations}, these products can be computed entirely
using fast operators such as the FFT. Similar efficiencies can be used
to compute a subgradient $g$ from the forward map
$\Ascr(U_{1} T U_{1}^{*})$.

For large problems, further efficiencies can be obtained simply by
computing a single eigenvector $u_{1}$, i.e., any unit-norm vector in
the range of $U_{1}$. In our implementation, we typically request at
least \emph{two} rightmost eigenpairs: this gives us an opportunity to
detect if the leading eigenpair is isolated. If it is, then the
subdifferential contains only a single element, which implies that $f$
is differentiable at that point.

Any sequence of step lengths $\{\alpha\k\}$ that satisfies the generic
conditions
\[
 \lim_{k\to\infty}\alpha\k=0,
 \quad
 \sum_{k=0}^{\infty}\alpha\k = \infty
\]
is sufficient to guarantee that the value of the objective at $y\k$
converges to the optimal value \cite[Proposition
3.2.6]{bertsekas2015}. A typical choice is $\alpha\k=\Oscr(1/k)$. Our
implementation defaults to a Barzilai-Borwein steplength
\citep{BarzBorw:1988} with a nonmonotonic linesearch
\citep{doi:10.1137/S1052623403428208} if it is detected that a sequence
of iterates is differentiable (by observing separation of the leading
eigenpair); and otherwise it falls back to a decreasing step size.

The projection operator $\Pscr$ onto the dual-feasible
set~\eqref{eq:gaugedual} is inexpensive when the residual is measured
in the 2-norm. In particular, if $\epsilon=0$, the dual-feasible set
is a halfspace, and the projection can be accomplished in linear
time. When $\epsilon$ is positive, the projection requires computing
the roots of a 1-dimensional degree-4 polynomial, which in practice
requires little additional time.

\subsection{Primal recovery} \label{sec:primal-recovery}

At each iteration of the descent method~\eqref{eq:6} for the
eigenvalue optimization problem~\eqref{eq:gaugedual}, we compute a
corresponding primal estimate
\begin{equation} \label{eq:8}
 X_{+} = U_{1}S_{+}U_{1}^{*}
\end{equation}
maintained in factored form.  The matrix $U_{1}$ has already been
computed in the evaluation of the objective and its subgradient;
see~\eqref{eq:7}. The positive semidefinite matrix $S_{+}$ is the
solution of the primal-recovery problem~\eqref{eq:primal-recovery}.

A byproduct of the primal-recovery problem is that it provides a
suitable stopping criterion for the overall algorithm. Because the
iterations $y\k$ are dual feasible, it follows from
Corollary~\ref{corl:pfd} that if~\eqref{eq:primal-recovery} has a zero
residual, then the dual iterate $y\k$ and the corresponding primal
iterate $X\k$ are optimal. Thus, we use the size of the residual to
determine a stopping test for approximate optimality.

\subsection{Primal-dual refinement} \label{sec:pr-du-refinement}

The primal-recovery procedure outlined in section~\ref{sec:primal-recovery}
is used only as a stopping criterion, and does not directly affect the
sequence of dual iterates from~\eqref{eq:6}.  In our numerical
experiments, we find that significant gains can be had by refining the
primal estimate~\eqref{eq:8} and feeding it back into the dual
sequence. We use the following procedure, which involves two auxiliary
subproblems that add relatively little to the overall cost.

The first step is to refine the primal estimate obtained
via~\eqref{eq:primal-recovery} by using its solution to determine the
starting point $Z_{0}=U_{1}S_{+}^{1/2}$ for the smooth unconstrained
non-convex problem
\begin{equation} \label{eq:10}
 \minimize{Z\in\Complex^{n\times r}}
 \quad
 h(Z):=\tfrac14\norm{\Ascr(ZZ^{*})-b_{\epsilon}}^{2}.
\end{equation}
In effect, we continue to minimize~\eqref{eq:primal-recovery}, where
additionally $U_{1}$ is allowed to vary. Several options are available
for solving this smooth unconstrained problem. Our implementation has
the option of using a steepest-descent iteration with a spectral
steplength and non-monotone linesearch
\citep{doi:10.1137/S1052623403428208}, or a limited-memory BFGS
method \cite[Section 7.2]{NoceWrig:2006}. The main cost at each iteration is
the evaluation of the gradient
\begin{equation}\label{eq:14}
 \nabla h(Z) = \Ascr^{*}(\Ascr(ZZ^*)-b_{\epsilon})Z.
\end{equation}
We thus obtain a candidate improved primal estimate
$\Xhat=\Zhat\Zhat^{*}$, where $\Zhat$ is a solution
of~\eqref{eq:10}. When $\epsilon=0$, this non-convex problem coincides
with the problem used by \cite{7029630}. They use the initialization
$Z_{0}=\gamma u_{1}$, where $u_{1}$ is a leading eigenvector of
$\Ascr^{*}b$, and
$\gamma=n\sum_{i}b_{i}/\sum_{i}\norm{a_{i}}^{2}$. Our initialization,
on the other hand, is based on a solution of the primal-recovery
problem~\eqref{eq:primal-recovery}.

The second step of the refinement procedure is to construct a
candidate dual estimate $\yhat$ from a solution of the constrained
linear-least-squares problem
\begin{equation}\label{eq:9}
\minimize{y\in\Real^{m}} \quad \half\norm{(\Ascr^{*}y)\Zhat-\lambdahat\Zhat}^{2}
\quad\st\quad \ip b y-\epsilon\norm{y}_{*}\ge1,
\end{equation}
where $\lambdahat:=1/\trace\Xhat\equiv1/\norm{\Zhat}^{2}_{F} $ is the
reciprocal of the primal objective value associated with $\Xhat$.
This constrained linear-least-squares problem attempts to construct a
vector $\yhat$ such that the columns of $\Zhat$ correspond to eigenvectors
of $\Ascr^{*}\yhat$ associated with $\lambdahat$. If
$f(\yhat) < f(y_{+})$, then $\yhat$ improves on the current dual
iterate $y_{+}$ obtained by the descent method~\eqref{eq:6}, and we
are free to use $\yhat$ in its place. This improved estimate, which is
exogenous to the dual descent method, can be considered a ``spacer''
iterate, as described by
\citet[Proposition~1.2.6]{Bert:1999}. Importantly, it does not
interfere with the convergence of the underlying descent method.  The
projected-descent method used to solve the dual sequence can also be
applied to~\eqref{eq:9}, though in this case the objective is
guaranteed to be differentiable.

\subsection{Algorithm summary} \label{sec:algo-summary}

The following steps summarize one iteration of the dual-descent
algorithm: $y$ is the current dual iterate, and $y^+$ is the updated
iterate. The primal iterate $X$ is maintained in factored form. Steps
5-7 implement the primal-dual refinement strategy described in
section~\ref{sec:pr-du-refinement}, and constitute a heuristic that may
improve the performance of the dual descent algorithm without
sacrificing convergence guarantees.

\vspace{-.5\baselineskip}
\RestyleAlgo{boxed}
\begin{algorithm}
  \DontPrintSemicolon
  \nl$(\lambda_1, U_1) \gets\ \lambda_1(\Ascr^* y)$
  \Comment*{eigenvalue computation}

  \nl \makebox[3ex][l]{$g$} $\gets\ \Ascr(U_1 T U^{*}_1)$
  \Comment*{gradient of dual objective; cf.~\eqref{eq:7}}
  
  \nl \makebox[3ex][l]{$y^+$} $\gets\ \Pscr(y - \alpha g)$
  \Comment*{projected subgradient step; cf.~\eqref{eq:6}}

  \nl \makebox[3ex][l]{$S_+$} $\gets\ $ solution of \eqref{eq:primal-recovery}
  \Comment*{primal recovery subproblem}

  \nl \makebox[3ex][l]{$\Zhat$} $\gets\ $ solution of \eqref{eq:10} initialized with $S_+$
  \Comment*{primal refinement}

  \nl \makebox[3ex][l]{$\yhat$} $\gets\ $ solution of \eqref{eq:9} initialized with $\Zhat$
  \Comment*{dual refinement}

  \If{$\lambda_1(\Ascr^*\yhat)<\lambda_1$} 
  {\nl$y^+\gets\yhat$ \Comment*{spacer step}}
\end{algorithm}
\vspace{-.5\baselineskip}

\noindent In Step~1, the rightmost eigenpair $(\lambda_1, U_1)$ of $\Ascr^* y$
is computed. The eigenvectors in the matrix $U_1$ are used in Step~2
to compute a subgradient $g$ for the dual objective. Any PSD matrix
$T$ that has trace equal to~1 can be used in Step~2. For example, the
case where only a single rightmost eigenvector $u_1$ can be afforded
corresponds to setting $T$ so that $U_1 TU_1^*=u_1u_1^*$. Step~3 is a
projected subgradient iteration with steplength $\alpha$. Step~4
solves the primal-recovery problem to determine the matrix $S_+$ used
to define a primal estimate $X_+=U_1 S_+ U_1^*$;
cf.~\eqref{eq:pdf-opt}. We use the factorization
$Z_0:=U_1 S_+^{\nicefrac12}$ to initialize the algorithm in the next
step.  Step~5 applies an algorithm to the nonlinear least-squares
problem~\eqref{eq:10} to obtain a stationary point $\Zhat$ used to
define the dual-refinement problem used in the next step. Step~6
computes a candidate dual solution $\yhat$ that---if it improves the
dual objective---is used to replace the latest dual estimate $y_+$.

\section{Numerical experiments} \label{sec:experiments}

This section reports on a set of numerical experiments for solving
instances of the phase retrieval and blind deconvolution problems
described in section~\ref{sec:formulations}. The various algorithmic pieces
described in section~\ref{sec:implementation} have been implemented as a
\Matlab\ software package. The implementation uses \Matlab's
\texttt{eigs} routine for the eigenvalue computations described in
section~\ref{sec:dual-descent}. We implemented a projected gradient-descent
method, which is used for solving~\eqref{eq:6}, \eqref{eq:10}, and~\eqref{eq:9}.

\subsection{Phase recovery} \label{ssec:plexp}

We conduct three experiments for phase retrieval via the PhaseLift
formulation. The first experiment is for a large collection of small
one-dimensional random signals, and is meant to contrast the approach
against a general-purpose convex optimization algorithm and a
specialized non-convex approach.  The second experiment tests problems
where the vector of observations $b$ is contaminated by noise, hence
testing the case where $\epsilon>0$. The third experiment tests the
scalability of the approach on a large two-dimensional natural image.

Our phase retrieval experiments follow the approach outlined
in~\cite{7029630}. The diagonal matrices
$C\k\in\Complex^{n\times n}$ encode diffraction patterns that
correspond to the $k$th ``mask'' ($k=1,\ldots,L$) through which a
signal $x_{0}\in\Complex^{n}$ is measured. The measurements are given
by
\[
b = \Ascr(x_{0}x_{0}^{*}):=\diag\left[
\begin{pmatrix}FC_{1}\\\vdots\\FC\L\end{pmatrix}
(x_{0}x_{0}^{*})
\begin{pmatrix}FC_{1}\\\vdots\\FC\L\end{pmatrix}^{\!\!*\,}
\right],
\]
where $F$ is the unitary discrete Fourier transform (DFT).  The
adjoint of the associated linear map $\Ascr$ is then
\[
 \Ascr^{*}y := \sum_{k=1}^{L} C_k^* F^* \Diag(y_k) F C_k,
\]
where $y=(y_{1},\ldots,y_{L})$ and $\Diag(y\k)$ is the diagonal matrix
formed from the vector $y\k$. The main cost in the evaluation of the
forward map $\Ascr(VV^{*})$ involves $L$ applications of the DFT for
each column of $V$. Each evaluation of the adjoint map applied to a
vector $v$---i.e., $(\Ascr^{*}y)v$---requires $L$ applications of both
the DFT and its inverse. In the experimental results reported below,
the columns labeled ``nDFT'' indicate the total number of DFT
evaluations used over the course of a run. The costs of these DFT
evaluations are invariant across the different algorithms, and
dominate the overall computation.

\subsubsection{Random Gaussian signals} \label{sec:random-gaussian-experiments}

In this section we consider a set of experiments for different numbers
of masks. For each value of $L=6,7,\ldots,12$, we generate a fixed set
of~100 random complex Gaussian vectors $x_{0}$ of length $n=128$, and
a set of~$L$ random complex Gaussian masks $C_{k}$.

Table~\ref{tab:phaselift-random} summarizes the results of applying
four different solvers to each set of 100 problems. The solver
\texttt{GAUGE} is our implementation of the approach summarized in
section~\ref{sec:algo-summary}; \texttt{TFOCS} \citep{BeckBobCandGrant:2011}
is a first-order conic solver applied to the primal
problem~\eqref{eq:sdp-primal}.  The version used here was modified to
avoid explicitly forming the matrix $\Ascr^{*}y$
\citep{StrohmerPrivate:2013}. The algorithm \texttt{WFLOW}
\citep{7029630} is a non-convex approach that attempts to recover the
original signal directly from the feasibility problem~\eqref{eq:10},
with $\epsilon=0$. To make sensible performance comparisons to
\texttt{WFLOW}, we add to its implementation a stopping test based on
the norm of the gradient~\eqref{eq:14}; the default algorithm
otherwise uses a fixed number of iterations.

\begin{table}[tb]
  \caption{Phase retrieval comparisons for random complex Gaussian signals of
    size $n=128$ measured using random complex Gaussian masks.
    Numbers of the form $n_{-e}$ are a shorthand for
    $n\cdot10^{-e}$.}
  \label{tab:phaselift-random}
\begin{center}\small
\begin {tabular}{rcr<{\pgfplotstableresetcolortbloverhangright }@{}l<{\pgfplotstableresetcolortbloverhangleft }cr<{\pgfplotstableresetcolortbloverhangright }@{}l<{\pgfplotstableresetcolortbloverhangleft }cr<{\pgfplotstableresetcolortbloverhangright }@{}l<{\pgfplotstableresetcolortbloverhangleft }rcr<{\pgfplotstableresetcolortbloverhangright }@{}l<{\pgfplotstableresetcolortbloverhangleft }r}%
\toprule & \multicolumn {3}{c}{\tt GAUGE} & \multicolumn {3}{c}{\tt GAUGE-feas} & \multicolumn {4}{c}{\tt TFOCS} & \multicolumn {4}{c}{\tt WFLOW} \\ \cmidrule (lr){2 -4} \cmidrule (lr){ 5- 7} \cmidrule (lr){8-11} \cmidrule (l ){12-15}$L$&nDFT&\multicolumn {2}{c}{xErr}&nDFT&\multicolumn {2}{c}{xErr}&nDFT&\multicolumn {2}{c}{xErr}&\%&nDFT&\multicolumn {2}{c}{xErr}&\%\\\midrule %
\pgfutilensuremath {12}&\pgfutilensuremath {18{,}330}&$1.6$&$_{-6}$&\pgfutilensuremath {3{,}528}&$1.3$&$_{-6}$&\pgfutilensuremath {2{,}341{,}800}&$3.6$&$_{-3}$&\pgfutilensuremath {100}&\pgfutilensuremath {5{,}232}&$1.2$&$_{-5}$&\pgfutilensuremath {100}\\%
\pgfutilensuremath {11}&\pgfutilensuremath {19{,}256}&$1.5$&$_{-6}$&\pgfutilensuremath {3{,}344}&$1.4$&$_{-6}$&\pgfutilensuremath {2{,}427{,}546}&$4.3$&$_{-3}$&\pgfutilensuremath {100}&\pgfutilensuremath {4{,}906}&$1.6$&$_{-5}$&\pgfutilensuremath {100}\\%
\pgfutilensuremath {10}&\pgfutilensuremath {19{,}045}&$1.4$&$_{-6}$&\pgfutilensuremath {3{,}120}&$1.6$&$_{-6}$&\pgfutilensuremath {2{,}857{,}650}&$5.5$&$_{-3}$&\pgfutilensuremath {100}&\pgfutilensuremath {4{,}620}&$2.1$&$_{-5}$&\pgfutilensuremath {100}\\%
\pgfutilensuremath {9}&\pgfutilensuremath {21{,}933}&$1.6$&$_{-6}$&\pgfutilensuremath {2{,}889}&$1.4$&$_{-6}$&\pgfutilensuremath {1.2\cdot 10^{7}}&$7.5$&$_{-3}$&\pgfutilensuremath {89}&\pgfutilensuremath {4{,}374}&$2.5$&$_{-5}$&\pgfutilensuremath {100}\\%
\pgfutilensuremath {8}&\pgfutilensuremath {23{,}144}&$2.1$&$_{-6}$&\pgfutilensuremath {2{,}688}&$1.9$&$_{-6}$&\pgfutilensuremath {1.1\cdot 10^{7}}&$1.2$&$_{-2}$&\pgfutilensuremath {22}&\pgfutilensuremath {4{,}080}&$3.3$&$_{-5}$&\pgfutilensuremath {100}\\%
\pgfutilensuremath {7}&\pgfutilensuremath {25{,}781}&$1.8$&$_{-6}$&\pgfutilensuremath {2{,}492}&$2.0$&$_{-6}$&\pgfutilensuremath {6{,}853{,}245}&$2.4$&$_{-2}$&\pgfutilensuremath {0}&\pgfutilensuremath {3{,}836}&$5.2$&$_{-5}$&\pgfutilensuremath {95}\\%
\pgfutilensuremath {6}&\pgfutilensuremath {34{,}689}&$3.0$&$_{-6}$&\pgfutilensuremath {2{,}424}&$2.5$&$_{-6}$&\pgfutilensuremath {2{,}664{,}126}&$6.4$&$_{-2}$&\pgfutilensuremath {0}&\pgfutilensuremath {3{,}954}&$9.5$&$_{-5}$&\pgfutilensuremath {62}\\\bottomrule %
\end {tabular}%

\end{center}
\end{table}

\begin{table}[tb]
  \caption{Additional comparisons for the random examples of Table~\ref{tab:phaselift-random}.}
  \label{tab:phaselift-random-nodfp}
\centering\small
\begin {tabular}{rcr<{\pgfplotstableresetcolortbloverhangright }@{}l<{\pgfplotstableresetcolortbloverhangleft }cr<{\pgfplotstableresetcolortbloverhangright }@{}l<{\pgfplotstableresetcolortbloverhangleft }}%
\toprule & \multicolumn {3}{c}{\tt GAUGE} & \multicolumn {3}{c}{\tt GAUGE-nodfp} \\ \cmidrule (lr){2-4} \cmidrule (lr){5-7}$L$&nDFT&\multicolumn {2}{c}{xErr}&nDFT&\multicolumn {2}{c}{xErr}\\\midrule %
\pgfutilensuremath {12}&\pgfutilensuremath {18{,}330}&$1.6$&$_{-6}$&\pgfutilensuremath {277{,}722}&$1.6$&$_{-6}$\\%
\pgfutilensuremath {11}&\pgfutilensuremath {19{,}256}&$1.5$&$_{-6}$&\pgfutilensuremath {314{,}820}&$1.6$&$_{-6}$\\%
\pgfutilensuremath {10}&\pgfutilensuremath {19{,}045}&$1.4$&$_{-6}$&\pgfutilensuremath {374{,}190}&$2.0$&$_{-6}$\\%
\pgfutilensuremath {9}&\pgfutilensuremath {21{,}933}&$1.6$&$_{-6}$&\pgfutilensuremath {485{,}658}&$1.9$&$_{-6}$\\%
\pgfutilensuremath {8}&\pgfutilensuremath {23{,}144}&$2.1$&$_{-6}$&\pgfutilensuremath {808{,}792}&$1.9$&$_{-6}$\\%
\pgfutilensuremath {7}&\pgfutilensuremath {25{,}781}&$1.8$&$_{-6}$&\pgfutilensuremath {2{,}236{,}885}&$2.3$&$_{-6}$\\%
\pgfutilensuremath {6}&\pgfutilensuremath {34{,}689}&$3.0$&$_{-6}$&\pgfutilensuremath {14{,}368{,}437}&$2.9$&$_{-6}$\\\bottomrule %
\end {tabular}%
\end{table}

We also show the results of applying the \texttt{GAUGE} code in a
``feasibility'' mode that exits as soon as the primal-refinment
subproblem (see Step~7 of the algorithm summary in
section~\ref{sec:algo-summary}) obtains a solution with a small
residual. This resulting solver is labeled \texttt{GAUGE-feas}. This
variant of \texttt{GAUGE} is in some respects akin to
\texttt{WFLOW}, with the main difference that \texttt{GAUGE-feas} uses
starting points generated by the dual-descent estimates, and generates
search directions and step-lengths for the feasibility problem from a
spectral gradient algorithm. The columns labeled ``xErr'' report the
median relative error
$\norm{x_{0}x_{0}^{*}-\xhat\xhat^{*}}\F/\norm{x_{0}}^{2}_{2}$ of the
100 runs, where $\xhat$ is the solution returned by the corresponding
solver. The columns labeled ``\%'' give the percentage of problems
solved to within a relative error of $10^{-2}$. At least on this set
of artificial experiments, the \texttt{GAUGE} solver (and its
feasibility variant \texttt{GAUGE-feas}) appear to be most
efficient. Table~\ref{tab:phaselift-random-nodfp} provides an
additional comparison of \texttt{GAUGE} with the variation
\texttt{GAUGE-nodfp}, which ignores the ``spacer'' iterate computed
by~\eqref{eq:9}. There seems to be significant practical benefit in
using the refined primal estimate to improve the dual sequence.  The
columns labeled ``\%'' are excluded for all versions of \texttt{GAUGE}
because these solvers obtained the prescribed accuracy for all
problems in each test set.

Note that the relative accuracy ``xErr'' is often slightly better for
\texttt{GAUGE-feas} than for \texttt{GAUGE}.  These small small
discrepancies are explained by the different stopping criteria between
the two versions of the solver. In particular, \texttt{GAUGE} will
continue iterating past the point at which \texttt{GAUGE-feas} 
normally terminates because it is searching for a dual certificate that
corresponds to the recovered primal estimate. This slightly changes
the computed subspaces $U_1$, which influence subsequent primal
estimates. Similar behaviour is exhbited in the noisy cases that we
consider in the next section.

As the number of measurements ($L$) decreases, we expect the problem
to be more difficult. Indeed, we can observe that the total amount of
work, as measured by the number of operator evaluations (i.e., the
ratio between \texttt{nDFT} and $L$), increases monotonically for all
variations of \texttt{GAUGE}.

\subsubsection{Random problems with noise}

In this set of experiments, we assess the effectiveness of the SDP
solver to problems with $\epsilon>0$, which could be useful in
recovering signals with noise.  For this purpose, it is convenient to
generate problems instances with noise and known primal-dual
solutions, which we can do by using Corollary~\ref{corl:pfd}. Each
instance is generated by first sampling octanary masks $C_k$---as
described by \cite{7029630}---and real Gaussian vectors $y\in\Real^m;$ a
solution $x_0\in\Complex^n$ is then chosen as a unit-norm rightmost
eigenvector of $\Ascr^*y$, and the measurements are computed as
$b:=\Ascr(x_0x_0^*)+\epsilon y/\norm{y},$ where $\epsilon$ is chosen
as $\epsilon:=\norm{b-\Ascr(x_0x_0^*)}=\eta\norm{b},$ for a given
noise-level parameter $\eta\in(0,1).$

\begin{table}[tb]
  \caption{Phase retrieval comparisons for problems with noise, i.e., $\epsilon>0$.
    Numbers of the form $n_{-e}$ are a shorthand for $n\cdot10^{-e}$. }
\label{tab:phaselift-random-noisy}
\centering\small
\begin {tabular}{rrrrrrrrrr}%
\toprule & & \multicolumn {2}{c}{\tt GAUGE} & \multicolumn {2}{c}{\tt GAUGE-nodfp} & \multicolumn {2}{c}{\tt GAUGE-feas} & \multicolumn {2}{c}{\tt WFLOW} \\ \cmidrule (lr){3 -4} \cmidrule (lr){5- 6} \cmidrule (lr){7- 8} \cmidrule (lr){9-10}$L$&$\eta \quad $&nDFT&\%&nDFT&\%&nDFT&\%&nDFT&\%\\\midrule %
\pgfutilensuremath {12}&\pgfutilensuremath {0.1}\%&\pgfutilensuremath {4{,}584}&\pgfutilensuremath {100}&\pgfutilensuremath {29{,}988}&\pgfutilensuremath {100}&\pgfutilensuremath {936}&\pgfutilensuremath {100}&\pgfutilensuremath {14{,}856}&\pgfutilensuremath {100}\\%
\pgfutilensuremath {9}&\pgfutilensuremath {0.1}\%&\pgfutilensuremath {3{,}222}&\pgfutilensuremath {100}&\pgfutilensuremath {36{,}292}&\pgfutilensuremath {100}&\pgfutilensuremath {774}&\pgfutilensuremath {100}&\pgfutilensuremath {11{,}511}&\pgfutilensuremath {100}\\%
\pgfutilensuremath {6}&\pgfutilensuremath {0.1}\%&\pgfutilensuremath {2{,}232}&\pgfutilensuremath {100}&\pgfutilensuremath {50{,}235}&\pgfutilensuremath {100}&\pgfutilensuremath {612}&\pgfutilensuremath {100}&\pgfutilensuremath {8{,}922}&\pgfutilensuremath {98}\\%
\pgfutilensuremath {12}&\pgfutilensuremath {0.5}\%&\pgfutilensuremath {3{,}768}&\pgfutilensuremath {100}&\pgfutilensuremath {27{,}252}&\pgfutilensuremath {100}&\pgfutilensuremath {936}&\pgfutilensuremath {100}&\pgfutilensuremath {14{,}808}&\pgfutilensuremath {100}\\%
\pgfutilensuremath {9}&\pgfutilensuremath {0.5}\%&\pgfutilensuremath {2{,}934}&\pgfutilensuremath {100}&\pgfutilensuremath {31{,}032}&\pgfutilensuremath {100}&\pgfutilensuremath {774}&\pgfutilensuremath {100}&\pgfutilensuremath {11{,}430}&\pgfutilensuremath {100}\\%
\pgfutilensuremath {6}&\pgfutilensuremath {0.5}\%&\pgfutilensuremath {2{,}148}&\pgfutilensuremath {100}&\pgfutilensuremath {38{,}766}&\pgfutilensuremath {100}&\pgfutilensuremath {606}&\pgfutilensuremath {100}&\pgfutilensuremath {8{,}790}&\pgfutilensuremath {98}\\%
\pgfutilensuremath {12}&\pgfutilensuremath {1.0}\%&\pgfutilensuremath {3{,}744}&\pgfutilensuremath {100}&\pgfutilensuremath {22{,}620}&\pgfutilensuremath {97}&\pgfutilensuremath {936}&\pgfutilensuremath {100}&\pgfutilensuremath {14{,}712}&\pgfutilensuremath {100}\\%
\pgfutilensuremath {9}&\pgfutilensuremath {1.0}\%&\pgfutilensuremath {2{,}934}&\pgfutilensuremath {100}&\pgfutilensuremath {24{,}813}&\pgfutilensuremath {96}&\pgfutilensuremath {774}&\pgfutilensuremath {100}&\pgfutilensuremath {11{,}331}&\pgfutilensuremath {99}\\%
\pgfutilensuremath {6}&\pgfutilensuremath {1.0}\%&\pgfutilensuremath {1\cdot 10^{7}}&\pgfutilensuremath {97}&\pgfutilensuremath {2\cdot 10^{5}}&\pgfutilensuremath {98}&\pgfutilensuremath {600}&\pgfutilensuremath {53}&\pgfutilensuremath {8{,}634}&\pgfutilensuremath {8}\\%
\pgfutilensuremath {12}&\pgfutilensuremath {5.0}\%&\pgfutilensuremath {95{,}952}&\pgfutilensuremath {26}&\pgfutilensuremath {9\cdot 10^{5}}&\pgfutilensuremath {90}&\pgfutilensuremath {936}&\pgfutilensuremath {0}&\pgfutilensuremath {14{,}148}&\pgfutilensuremath {0}\\%
\pgfutilensuremath {9}&\pgfutilensuremath {5.0}\%&\pgfutilensuremath {2\cdot 10^{6}}&\pgfutilensuremath {89}&\pgfutilensuremath {8\cdot 10^{5}}&\pgfutilensuremath {90}&\pgfutilensuremath {774}&\pgfutilensuremath {0}&\pgfutilensuremath {10{,}701}&\pgfutilensuremath {0}\\%
\pgfutilensuremath {6}&\pgfutilensuremath {5.0}\%&\pgfutilensuremath {2\cdot 10^{6}}&\pgfutilensuremath {82}&\pgfutilensuremath {5\cdot 10^{5}}&\pgfutilensuremath {98}&\pgfutilensuremath {600}&\pgfutilensuremath {0}&\pgfutilensuremath {7{,}728}&\pgfutilensuremath {0}\\%
\pgfutilensuremath {12}&\pgfutilensuremath {10.0}\%&\pgfutilensuremath {1\cdot 10^{5}}&\pgfutilensuremath {17}&\pgfutilensuremath {1\cdot 10^{6}}&\pgfutilensuremath {89}&\pgfutilensuremath {912}&\pgfutilensuremath {0}&\pgfutilensuremath {13{,}548}&\pgfutilensuremath {0}\\%
\pgfutilensuremath {9}&\pgfutilensuremath {10.0}\%&\pgfutilensuremath {8\cdot 10^{5}}&\pgfutilensuremath {78}&\pgfutilensuremath {7\cdot 10^{5}}&\pgfutilensuremath {91}&\pgfutilensuremath {765}&\pgfutilensuremath {0}&\pgfutilensuremath {10{,}125}&\pgfutilensuremath {0}\\%
\pgfutilensuremath {6}&\pgfutilensuremath {10.0}\%&\pgfutilensuremath {7\cdot 10^{5}}&\pgfutilensuremath {90}&\pgfutilensuremath {5\cdot 10^{5}}&\pgfutilensuremath {100}&\pgfutilensuremath {588}&\pgfutilensuremath {0}&\pgfutilensuremath {7{,}098}&\pgfutilensuremath {0}\\%
\pgfutilensuremath {12}&\pgfutilensuremath {50.0}\%&\pgfutilensuremath {2\cdot 10^{5}}&\pgfutilensuremath {53}&\pgfutilensuremath {5\cdot 10^{5}}&\pgfutilensuremath {94}&\pgfutilensuremath {888}&\pgfutilensuremath {0}&\pgfutilensuremath {11{,}424}&\pgfutilensuremath {0}\\%
\pgfutilensuremath {9}&\pgfutilensuremath {50.0}\%&\pgfutilensuremath {1\cdot 10^{5}}&\pgfutilensuremath {42}&\pgfutilensuremath {3\cdot 10^{5}}&\pgfutilensuremath {99}&\pgfutilensuremath {738}&\pgfutilensuremath {0}&\pgfutilensuremath {8{,}586}&\pgfutilensuremath {0}\\%
\pgfutilensuremath {6}&\pgfutilensuremath {50.0}\%&\pgfutilensuremath {1\cdot 10^{5}}&\pgfutilensuremath {24}&\pgfutilensuremath {2\cdot 10^{5}}&\pgfutilensuremath {95}&\pgfutilensuremath {588}&\pgfutilensuremath {0}&\pgfutilensuremath {7{,}176}&\pgfutilensuremath {0}\\\bottomrule %
\end {tabular}%

\end{table}

For these experiments, we generate 100 instances with $n=128$ for each
pairwise combination $(L,\eta)$ with $L\in\{6,9,12\}$ and
$\eta\in\{0.1\%,0.5\%,1\%,5\%,10\%,50\%\}.$
Table~\ref{tab:phaselift-random-noisy} summarizes the results of
applying the three variations of \texttt{GAUGE}, and the
\texttt{WFLOW} solver, to these problems. It is not clear that
\texttt{GAUGE-feas} and \texttt{WFLOW} are relevant for this
experiment, but for interest we include them in the results. As with
the experiments in section~\ref{sec:random-gaussian-experiments}, a solve is
``successful'' if it recovers the true solution with a relative error
of $10^{-2}$. The median relative error for all solvers is comparable,
and hence we omit the column ``xErr''. \texttt{GAUGE-nodfp}
is generally successful in recovering the rank-1 minimizer for most
problems---even for cases with significant noise, though in these
cases the overall cost increases considerably. On the other hand,
\texttt{GAUGE} is less successful: it appears that although the
primal-dual refinement procedure can help to reduce the cost of
successful recovery in low-noise settings, in high-noise settings it
may obtain primal solutions that are not necessarily close to the true
signal. For noise levels over 5\%, \texttt{GAUGE-feas} and
\texttt{WFLOW} are unable to recover a solution within the prescribed
accuracy, which points to the benefits of the additional cost of
obtaining a primal-dual optimal point, rather than just a primal
feasible point.

\subsubsection{Two-dimensional signal}

We conduct a second experiment on a stylized application in order to
assess the scalability of the approach to larger problem sizes. In
this case the measured signal $x_{0}$ is a two-dimensional image of
size $1600\times 1350$ pixels, shown in Figure~\ref{fig:nebula}, which
corresponds to $n=2.2\cdot10^{6}$. The size of the lifted formulation
is on the order of $n^{2}\approx10^{12}$, which makes it clear that
the resulting SDP is enormous, and must be handled by a specialized
solver. We have excluded \texttt{TFOCS} from the list of candidate
solvers because it cannot make progress on this example. We generate
10 and 15 octanary masks. Table~\ref{tab:2-d-image}
summarizes the results. The column headers carry the same meaning as
Table~\ref{tab:phaselift-random}.

\begin{figure}[tb]
  \centering
  \includegraphics[width=.6\textwidth]{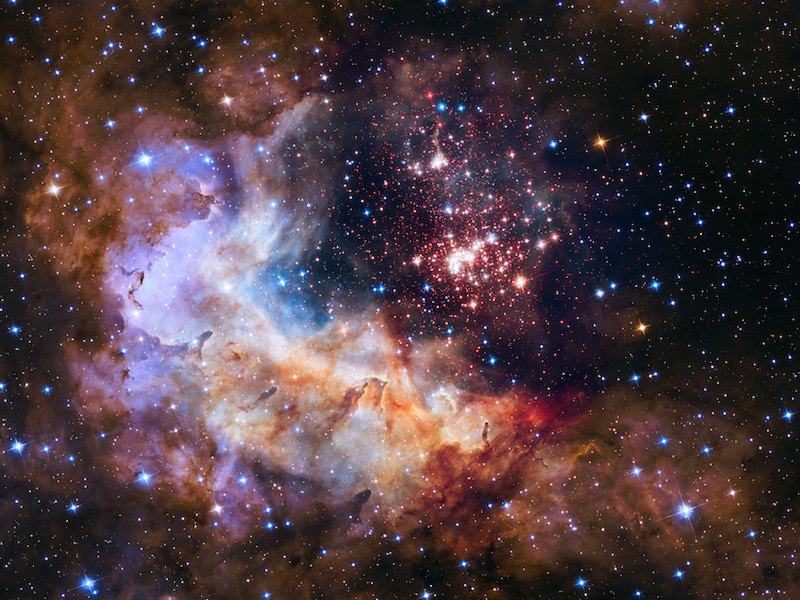}
  \caption{Image used for phase retrieval experiment; size $1600\times
  1350$ pixels (7.5MB).}
  \label{fig:nebula}
\end{figure}
\begin{table}[t]
\caption{Phase retrieval comparisons on a two-dimensional image.}
\label{tab:2-d-image}
\centering\small
\begin {tabular}{rcccccc}%
\toprule & \multicolumn {2}{c}{\tt GAUGE} & \multicolumn {2}{c}{\tt GAUGE-feas} & \multicolumn {2}{c}{\tt WFLOW} \\ \cmidrule (lr){2 -3} \cmidrule (lr){4- 5} \cmidrule (lr){6- 7}$L$&nDFT&xErr&nDFT&xErr&nDFT&xErr\\\midrule %
\pgfutilensuremath {15}&\pgfutilensuremath {200{,}835}&\pgfutilensuremath {2.1\cdot 10^{-6}}&\pgfutilensuremath {5{,}700}&\pgfutilensuremath {2.1\cdot 10^{-6}}&\pgfutilensuremath {8{,}100}&\pgfutilensuremath {4.1\cdot 10^{-6}}\\%
\pgfutilensuremath {10}&\pgfutilensuremath {195{,}210}&\pgfutilensuremath {5.8\cdot 10^{-7}}&\pgfutilensuremath {12{,}280}&\pgfutilensuremath {9.1\cdot 10^{-7}}&\pgfutilensuremath {12{,}340}&\pgfutilensuremath {2.1\cdot 10^{-5}}\\\bottomrule %
\end {tabular}%

\end{table}

\subsection{Blind deconvolution} \label{ssec:bcsexp}

In this blind deconvolution experiment, the convolution of two signals
$s_{1}\in\Complex^{m}$ and $s_{2}\in\Complex^{m}$ are measured. Let
$B_{1}\in\Complex^{m\times n_{1}}$ and
$B_{2}\in\Complex^{m\times n_{2}}$ be two bases. The circular
convolution of the signals can be described by
\begin{align*}
  b = s_{1}*s_{2}&= (B_1x_1)*(B_2x_2)
  \\&= F^{-1}\diag\big((FB_1x_1)(FB_2x_2)^T\big)
  \\&= F^{-1}\diag\big((FB_1)(x_1\conj{x}_2^*)(\conj{FB}_2)^*\big)
  =: \Ascr(x_{1}\conj{x}_{2}^{*}),
\end{align*}
where $\Ascr$ is the corresponding asymmetric linear map with the adjoint
\[
\Ascr^{*}y := (FB_1)^*\Diag(Fy)(\conj{FB_2}).
\]
Because $F$ is unitary, it is possible to work instead with
measurements
\[
\bhat \equiv Fb  = \diag\big((FB_1)(x_1\conj{x}_2^*)(\conj{FB_{2}})^*\big)
\]
in the Fourier domain. For the experiments that we run, we choose to
work with the former real-valued measurements $b$ because they do not
require accounting for the imaginary parts, and thus the number of
constraints in~\eqref{eq:11} that would be required otherwise is
reduced by half.

\begin{figure}[tb]
  \centering
  \begin{tabular}{@{}c@{}c@{}c@{}}
    \includegraphics[trim={16 32 48 32},clip,width=.32\textwidth]{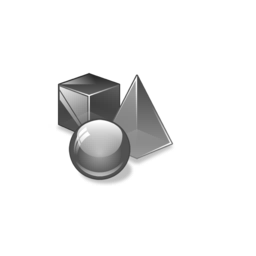}
   &\includegraphics[trim={32 32 32 32},clip,width=.32\textwidth]{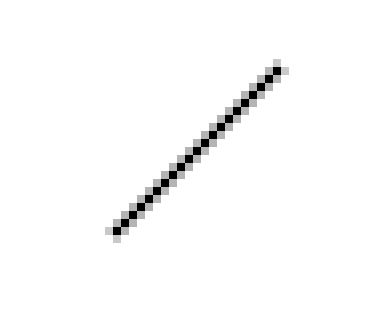}
   &\includegraphics[trim={16 32 48 32},clip,width=.32\textwidth]{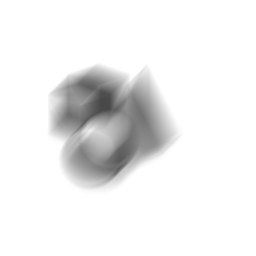}
 \\[-10pt](a)&(b)&(c)
  \end{tabular}
\\[12pt]
  \begin{tabular}{c@{\ }c@{\ }c@{}}
   \includegraphics[trim={16 32 48 32},clip,width=.32\textwidth]{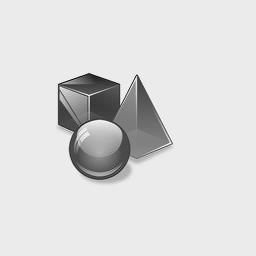}
  &\includegraphics[trim={16 32 48 32},clip,width=.32\textwidth]{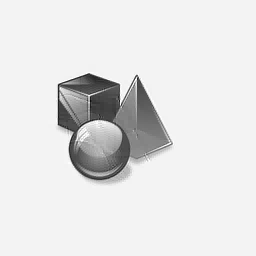}
  &\includegraphics[trim={16 32 48 32},clip,width=.32\textwidth]{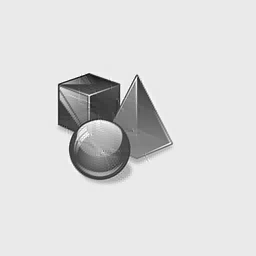}
  \\(d)&(e)&(f)
  \end{tabular}
  \caption{Images used for the blind deconvolution experiments: (a)
    original image; (b) zoom of the motion-blurring kernel; (c)
    blurred image; (d) image recovered by the augmented Lagrangian
    approach; (e) image recovered by \texttt{GAUGE}; (f) image
    recovered by \texttt{GAUGE-feas}. The shaded background on the
    recovered images is an artifact of the uniform scaling used to
    highlight any error between the original and recovered signals. }
  \label{fig:bd}
\end{figure}

We follow the experimental setup outlined by \cite{Ahmed:2014} and use
the data and code that they provide. In that setup, $B_{1}$ is a
subset of the Haar wavelet basis, and $B_{2}$ is a mask corresponding
to a subset of the identity basis. The top row of Figure~\ref{fig:bd}
shows the original image, the blurring kernel, and the observed
blurred image. The second row of the figure shows the image
reconstructed using the augmented Lagrangian code provided by
\cite{Ahmed:2014}, \texttt{GAUGE}, and
\texttt{GAUGE-feas}. Table~\ref{tab:bd} summarizes the results of
applying the three solvers. The columns headed ``nDFT'' and ``nDWT''
count the number of discrete Fourier and wavelet transforms required
by each solver; the columns headed ``xErr1'' and ``xErr2'' report the
relative errors $\norm{x_{i}-\xhat_{i}}_{2}/\norm{x_{i}}_{2}$,
$i=1,2$, where $\xhat_{i}$ are the recovered solutions; the column
headed ``rErr'' reports the relative residual error
$\norm{b-\Ascr(\xhat_{1}\conj{\xhat}^{*}_{2})}_{2}/\norm{b}_{2}$.
Although the non-convex augmented Lagrangian approach yields visibly
better recovery of the image, the table reveals that the solutions are
numerically similar, and are recovered with far less work.

\begin{table}[t]
\caption{Blind deconvolution comparisons.}
\label{tab:bd}
\centering\small
\begin {tabular}{lccccc}%
\toprule solver&nDFT&nDWT&xErr1&xErr2&rErr\\\midrule %
aug Lagrangian&\pgfutilensuremath {92{,}224}&\pgfutilensuremath {76{,}872}&\pgfutilensuremath {7.9\cdot 10^{-2}}&\pgfutilensuremath {5.0\cdot 10^{-1}}&\pgfutilensuremath {1.4\cdot 10^{-4}}\\%
GAUGE&\pgfutilensuremath {17{,}432}&\pgfutilensuremath {8{,}374}&\pgfutilensuremath {1.9\cdot 10^{-2}}&\pgfutilensuremath {5.4\cdot 10^{-1}}&\pgfutilensuremath {3.8\cdot 10^{-4}}\\%
GAUGE-feas&\pgfutilensuremath {4{,}128}&\pgfutilensuremath {2{,}062}&\pgfutilensuremath {8.4\cdot 10^{-2}}&\pgfutilensuremath {5.5\cdot 10^{-1}}&\pgfutilensuremath {4.0\cdot 10^{-4}}\\\bottomrule %
\end {tabular}%

\end{table}

\section{Extensions} \label{sec:extensions}

The problems~\eqref{eq:primal-probs} that we have considered so far
are stated in their simplest form. General semidefinite optimization
problems with nonnegative optimal value, and reweighted formulations
for rank minimization, as introduced by \cite{MohanFazel:2010} and
\cite{Candes:2013}, are also useful and can be accommodated by our
approach.

In the context of rank minimization over the PSD cone, an approximate
minimum-rank solution $\Xhat$ (e.g., computed via trace minimization)
might be used to obtain an even better approximation by using the
weighted objective $\ip{C}{X}$, where $C:=(\delta I+\Xhat)\inv$ and
$\delta$ is a small positive parameter. We might reasonably expect
that the objective value at a minimizer of this objective more closely
matches the rank function.  \cite{Candes:2013} show that such an
iteratively reweighted sequence of trace minimization problems can
improve the range of signals recoverable using PhaseLift.  Each
problem in that sequence uses the previous solution $\Xhat$ to derive
a weighting matrix $C=(\delta I+\Zhat\Zhat^*)\inv$ for the next
problem.  The inverse of the matrix $C$ is a low-rank update
$\Zhat\Zhat^*\approx\Xhat$ to a small regularizing multiple $\delta$
of the identity matrix.  This idea generalizes that of reweighting the
$1$-norm for cardinality minimization problems in compressed sensing,
where the number of nonzero entries of a vector $x$ is approximated by
$\sum_i\abs{x_i}/(\abs{\xhat_i}+\delta)$ for small $\delta$ and an
available approximation $\xhat$ (e.g., computed via $1$-norm
minimization).

In the next sections we derive the corresponding gauge duals for the weighted
formulations of both trace minimization in the PSD cone and nuclear-norm minimization.

\subsection{Nonnegative semidefinite optimization} \label{ssec:wtm}

Consider the semidefinite optimization problem
\begin{equation}
  \label{eq:wtm}
  \minimize{X\in\Hscr^n}
  \quad \ip{C}{X}
  \quad\st\quad
    \norm{b-\Ascr X}\le\epsilon,\ X\succeq0,
\end{equation}
where $C\succ0$.  
Define the maps
\[
  \Cscr(\cdot):=C^{-\frac{1}{2}}(\cdot)C^{-\frac{1}{2}}
  \text{and} \Ascr_C:=\Ascr\circ\Cscr^{-1}.
\]
It is evident that $X\succeq0$ if and only if $\Cscr(X)\succeq0$, and
so the SDP problem can be stated equivalently as
\begin{equation*}
  \minimize{X\in\Hscr^n}
  \quad \trace\, \Cscr(X)
  \quad\st\quad
    \norm{b-\Ascr_C(\Cscr(X))}\le\epsilon,\
    \Cscr(X)\succeq0.
\end{equation*}
Because $\Cscr$ is a bijection, we can optimize over
$\Xhat:=\Cscr(X)$ instead of $X$:
\begin{equation}\label{eq:wtmhat}
  \minimize{\Xhat\in\Hscr^n}
  \quad \trace\,\Xhat
  \quad\st\quad
    \norm{b-\Ascr_C\Xhat}\le\epsilon,\ \Xhat\succeq0.
\end{equation}
This clearly falls within the structure of~\eqref{eq:sdp-primal}, and
has the corresponding gauge dual
\begin{equation}\label{eq:wtmhatgd}
  \minimize{y\in\Real^m}
  \quad[\lambda_1(\Ascr_C^*y)]_+
  \quad\st\quad
    \ip b y-\epsilon\norm{y}_*\ge1.
\end{equation}
Observe that
$\lambda_1(\Ascr_C^*y)=\lambda_1(C^{-\frac{1}{2}}(\Ascr^*y)C^{-\frac{1}{2}})=\lambda_1(\Ascr^*y,C).$
Then
\begin{equation}
  \label{eq:wtmgd}
  \minimize{y\in\Real^m}
  \quad [\lambda_1(\Ascr^*y,C)]_+
  \quad\st\quad
    \ip b y - \epsilon\norm{y}_*\ge1.
\end{equation}

This shows that the introduction of a weighting matrix $C$ that is not a
simple multiple of the identity leads to a dual gauge problem
involving the minimization of the rightmost \emph{generalized}
eigenvalue of $\Ascr^*y$ with respect to that weight. Now that we have
a formulation for the gauge dual problem, we focus on how a primal
solution to the original weighted trace minimization can be computed
given a dual minimizer. This extends Corollary~\ref{corl:pfd}.
\begin{corollary}\label{corl:wtmpfd}
  Suppose that problem~\eqref{eq:wtm} is feasible and
  $0\le\epsilon<\norm{b}.$ Let $y\in\Real^m$ be an arbitrary optimal
  solution for the dual gauge~\eqref{eq:wtmgd}, $r_1\in\{1,\ldots,n\}$
  be the multiplicity of $\lambda_1(\Ascr^*y,C)$, and
  $U_1\in\Complex^{n\times r_1}$ be the matrix formed by the first
  $r_1$ generalized eigenvectors of $\Ascr^*y$ with respect to
  $C.$ Then $X\in\Hscr^n$ is a solution for the primal
  problem~\eqref{eq:wtm} if and only if there exists
  $S\succeq0$ such that
  \[
    X=U_1SU_1^* \text{and} (b-\Ascr
    X)\in\epsilon\partial\norm{\cdot}_*(y).
  \]

\end{corollary}
\begin{proof}
  A solution for~\eqref{eq:wtmgd} is clearly a solution
  for~\eqref{eq:wtmhatgd}. We may thus invoke
  Corollary~\ref{corl:pfd} and assert that $\Xhat\in\Hscr^n$ is a
  solution for~\eqref{eq:wtmhat} if and only if there is $S\succeq0$ such that
  $\Xhat=\Uhat_1S\Uhat_1^*$ and
  $(b-\Ascr_C\Xhat)\in\epsilon\partial\norm{\cdot}_*(y),$ where
  $\Uhat_1\in\Complex^{n\times r_1}$ is a matrix formed by the first
  $r_1$ eigenvectors of
  $\Ascr_C^*y=C^{-\frac{1}{2}}(\Ascr^*y)C^{-\frac{1}{2}}.$ From the
  structure of $\Cscr,$ we have that $X$ is a solution
  to~\eqref{eq:wtm} if and only if $X=\Cscr(\Xhat).$ Thus,
  $X=C^{-\frac12}\Uhat_1S\Uhat_1^*C^{-\frac12}=U_1SU_1^*,$ where
  $U_1:=C^{-{\frac12}}\Uhat_1$ corresponds to the first $r_1$
  generalized eigenvectors of $\Ascr^*y$ with respect to $C.$
\end{proof}

Once again, this provides us with a way to recover a solution to the
weighted trace minimization problem by computing a solution to the
gauge dual problem (now involving the rightmost generalized
eigenvalue) and then solving a problem of potentially much reduced
dimensionality.

\subsection{Weighted affine nuclear-norm optimization} \label{ssec:wnnm}

We can similarly extend the reweighted extension to the asymmetric
case~\eqref{eq:nuc-primal}.  Let $C_1\in\Hscr^{n_1}$ and
$C_2\in\Hscr^{n_2}$ be invertible. The weighted nuclear-norm
minimization problem becomes
\begin{equation}
  \label{eq:wnn}
  \minimize{X\in\Complex^{n_1\times n_2}}
  \quad \norm{C_1 X C_2^{*}}_1
  \quad\st\quad
    \norm{b-\Ascr X}\le\epsilon.
\end{equation}
Define the weighted quantities
\[
  \Cscr(\cdot)=C_1\inv(\cdot)C_2^{-*}:\Complex^{n_1\times
  n_2}\to\Complex^{n_1\times n_2},
  \quad
  \Ascr_C=\Ascr\circ\Cscr,
  \text{and}
  \Xhat:=\Cscr(X).
\]
The weighted problem can then be stated equivalently as
\begin{equation*}
  \minimize{\Xhat\in\Complex^{n_1\times n_2}}
  \quad \norm{\Xhat}_1
  \quad\st\quad
    \norm{b-\Ascr_C\Xhat}\le\epsilon,
\end{equation*}
which, following the approach introduced in~\cite{Fazel:2002}, can be
embedded in a symmetric problem:
\begin{equation}\label{eq:nnmsdp}
\begin{aligned}
  &\minimize{\substack{\Uhat\in\Hscr^{n_1}\\\Vhat\in\Hscr^{n_2}\\\Xhat\in\Complex^{n_1\times
        n_2}}} \quad
  \left\langle\frac{1}{2}I,\begin{pmatrix}\Uhat&\Xhat\\\Xhat^*&\Vhat\end{pmatrix}\right\rangle
  \\&\;\;\st\quad
  \begin{pmatrix}\Uhat&\Xhat\\\Xhat^*&\Vhat\end{pmatrix}\succeq0
  \text{ and }
  \norm{b-\Ascr_C\Xhat}\le\epsilon.
\end{aligned}
\end{equation}
Define the measurement operator from
$\Mscr:\Hscr^{n_1+n_2}\to\Complex^m$ by the map
\begin{align*}
  \begin{pmatrix}\Uhat&\Xhat\\\Xhat^*&\Vhat\end{pmatrix}&\mapsto\Ascr_C\Xhat,
\end{align*}
and identify $\Complex^m$ with $\Real^{2m}$ as a real inner-product
space. The adjoint of the measurement operator is then given by
\[
  \Mscr^*y
  =\begin{pmatrix}
    0&\Ascr_C^*y \\(\Ascr_C^*y)^{*}&0
  \end{pmatrix},\]
where $\Ascr_C^*y=\frac{1}{2}\sum_{i=1}^m C_1^{-1}A_iC_2^{-*}y_i$.  We can now
state the gauge dual problem:
\begin{align} \label{eq:nnmsdpgd}
  &\minimize{y\in\Complex^m}
  \quad \left[\lambda_1(\Mscr^{*}y,\half I) \right]_+
  \quad\st\quad
    \mathfrak{R}\ip b y - \epsilon\norm{y}_*\ge1.
\end{align}
Observe the identity
\begin{align*}
  \lambda_{1}\left(\Mscr^{*}y,\half I\right)
   &=\ \lambda_{1}(2\Mscr^{*}y)
  \\&=\left[\lambda_1\begin{pmatrix}0&\sum_{i=1}^mC_1^{-1}A_iC_2^{-*}y_i\\(\sum_{i=1}^mC_1^{-1}A_iC_2^{-*}y_i)^*&0\end{pmatrix}\right]_+
  \\&=\left[\norm{C_1^{-1}(\Ascr^*y)C_2^{-*}}_\infty\right]_+
     =\norm{C_1^{-1}(\Ascr^*y)C_2^{-*}}_\infty.
\end{align*}
We can now deduce the simplified form for the gauge dual problem:
\begin{equation}
  \label{eq:wnngd}
  \minimize{y\in\Complex^m}
  \quad \norm{C_1^{-1}(\Ascr^*y)C_2^{-*}}_\infty
  \quad\st\quad
    \mathfrak{R}\ip b y - \epsilon\norm{y}_*\ge1.
\end{equation}

This weighted gauge dual problem can be derived from first principles
using the tools from section~\ref{sec:gauge-duality} by observing that
the primal problem is already in standard gauge form. We chose this
approach, however, to make explicit the close connection between the
(weighted) nuclear-norm minimization problem and the (weighted)
trace-minimization problem described in section~\ref{ssec:wtm}.

The following result provides a way to characterize solutions of the
nuclear norm minimization problem when a solution to the dual gauge
problem is available.
\begin{corollary}\label{corl:nnmpfd}
  Suppose that problem~\eqref{eq:wnn} is feasible and
  $0\le\epsilon<\norm{b}$. Let $y\in\Complex^m$ be an arbitrary
  optimal solution for the dual gauge problem~\eqref{eq:wnngd},
  $r_1\in\{1,\ldots,n\}$ be the multiplicity of
  $\sigma_1(C_1^{-1}(\Ascr^*y)C_2^{-*}),$ $U_1\in\Complex^{n_1\times r_1}$ and
  $V_1\in\Complex^{n_2\times r_1}$ be the matrices formed by the first
  $r_1$ left  and right singular-vectors of $C_1^{-1}(\Ascr^*y)C_2^{-*},$
  respectively. Then $X\in\Complex^{n_1\times n_2}$ is a solution for
  the primal problem~\eqref{eq:wnn} if and only if there exists
  $S\succeq0$ such that $X=(C_1^{-1}U_1)S(C_2^{-1}V_1)^*$ and
  $(b-\Ascr X)\in\epsilon\partial\norm{\cdot}_*(y).$
\end{corollary}
\begin{proof}
  A solution for~\eqref{eq:wnngd} is clearly a solution
  for~\eqref{eq:nnmsdpgd}; this way we invoke
  Corollary~\ref{corl:pfd} and have that
  $(\Uhat,\Vhat,\Xhat)\in\Hscr^{n_1}\times\Hscr^{n_2}\times\Complex^{n_1\times
    n_2}$
  induce a solution for~\eqref{eq:nnmsdp} if and only if there is $S\succeq0$
  such that $\Xhat=\Uhat_1S\Vhat_1^*$ and
  $(b-\Ascr_C\Xhat)\in\epsilon\partial\norm{\cdot}_*(y),$ where
  $\Uhat_1\in\Complex^{n_1\times r_1}$ and
  $\Vhat_1\in\Complex^{n_2\times r_1}$ are matrices formed by the
  first $r_1$ left and right singular-vectors of
  $\Ascr_C^*y=C_1^{-1}(\Ascr^*y)C_2^{-*}.$ From the structure of $\Cscr,$ we
  have that $X$ is a solution to~\eqref{eq:wnn} if and only if $X=\Cscr(\Xhat).$
  This way, $X=(C_1^{-1}\Uhat_1)S(C_2^{-1}\Vhat_1)^*.$
\end{proof}

\section{Conclusions} \label{sec:conclusions}

The phase retrieval and blind deconvolution applications are examples
of convex relaxations of non-convex problems that give rise to large
spectral optimization problems with strong statistical guarantees for
correctly reconstructing certain signals. One of the criticisms that
have been leveled at these relaxation approaches is that they lead to
problems that are too difficult to be useful in practice. This has led
to work on non-convex recovery algorithms that may not have as-strong
statistical recovery guarantees, but are nonetheless effective in
practice; \cite{NIPS2013_5041,7029630,2015arXiv150607868W}.  Our
motivation is to determine whether it is possible to develop convex
optimization algorithms that are as efficient as non-convex
approaches.  The numerical experiments on these problems suggest that
the gauge-dual approach may prove effective. Indeed, other convex
optimization algorithms may be possible, and clearly the key to their
success will be to leverage the special structure of these problems.

A theoretical question we have not addressed is to delineate
conditions under which dual attainment will hold. In particular, the
conclusion~\eqref{eq:4} of Theorem~\ref{prop:strong-duality} is
asymmetric: we can assert that a primal solution exists that attains
the primal optimal value (because the Lagrange dual is strictly
feasible), but we cannot assert that a dual solution exists that
attains the dual optimal value. A related theoretical question is to
understand the relationship between the quality of suboptimal dual
solutions, and the quality of the primal estimate obtained by the
primal recovery procedure.

In our experiments, we have observed that the rightmost eigenvalue of
$\Ascr^{*}y$ remains fairly well separated from the others across
iterations. This seems to contribute to the overall effectiveness of
the dual-descent method.  Is there a special property of these
problems or of the algorithm that encourages this separation property?
It seems likely that there are solutions $y$ at which the objective is
not differentiable, and in that case, we wonder if there are
algorithmic devices that could be used to avoid such points.

The dual-descent method that we use to solve the dual subproblem
(cf.\@ section~\ref{sec:dual-descent}) is only one possible algorithm among
many. Other more specialized methods, such as the spectral bundle
method of \citet{HelmbergRendl:2000}, its second-order variant
\citep{helmberg2014spectral}, or the stochastic-gradient method of
\citet{doi:10.1137/12088728X}, may prove effective alternatives.

We have found it convenient to embed the nuclear-norm minimization
problem~\eqref{eq:nuc-primal} in the SDP formulation
\eqref{eq:sdp-primal} because it allows us to use the same solver for
both problems. Further efficiencies, however, may be gained by
implementing a solver that applied directly to the corresponding gauge
dual
\begin{equation*}
  \minimize{y\in\Complex^m}
  \quad \norm{\Ascr^*y}_\infty
  \quad\st\quad
    \mathfrak{R}\ip b y - \epsilon\norm{y}_*\ge1.
\end{equation*}
This would require an iterative solver for evaluating leading singular
values and singular vectors of the asymmetric operator
$\Ascr^{*}y$, such as PROPACK \citep{Lars:2001}.

\section*{Acknowledgments}

We extend sincere thanks to our colleague Nathan Krislock, who was
involved in an earlier incarnation of this project, and to our
colleague Ting Kei Pong, who was our partner in establishing the
crucial gauge duality theory in \cite*{FriedlanderMacedoPong:2014}. We
are also grateful to Xiaodong Li and Mahdi Soltanolkotabi for help
with using their \texttt{WFLOW} code. Finally, we wish to thank two
anonymous referees who provided a careful list of comments and
suggestions that helped to clarify our presentation.

\providecommand{\noopsort}[1]{}


\begin{thebibliography}{39}
\providecommand{\natexlab}[1]{#1}
\providecommand{\url}[1]{\texttt{#1}}
\expandafter\ifx\csname urlstyle\endcsname\relax
  \providecommand{\doi}[1]{doi: #1}\else
  \providecommand{\doi}{doi: \begingroup \urlstyle{rm}\Url}\fi

\bibitem[Ahmed et~al.(2014)Ahmed, Recht, and Romberg]{Ahmed:2014}
A.~Ahmed, B.~Recht, and J.~Romberg.
\newblock Blind deconvolution using convex programming.
\newblock \emph{IEEE Trans. Inform. Theory}, 60\penalty0 (3):\penalty0
  1711--1732, 2014.

\bibitem[{Aravkin} et~al.(2016){Aravkin}, {Burke}, {Drusvyatskiy},
  {Friedlander}, and {Roy}]{2016arXiv160201506A}
A.~Y. {Aravkin}, J.~V. {Burke}, D.~{Drusvyatskiy}, M.~P. {Friedlander}, and
  S.~{Roy}.
\newblock {Level-set methods for convex optimization}.
\newblock \emph{ArXiv e-prints}, Feb. 2016.

\bibitem[Argyriou et~al.(2014)Argyriou, Signoretto, and
  Suykens]{argyriou2014hybrid}
A.~Argyriou, M.~Signoretto, and J.~Suykens.
\newblock Hybrid conditional gradient-smoothing algorithms with applications to
  sparse and low rank regularization.
\newblock In \emph{Regularization, Optimization, Kernels, and Support Vector
  Machines}, chapter~3, page~53. CRC Press, 2014.

\bibitem[Bach(2015)]{bach2015duality}
F.~Bach.
\newblock Duality between subgradient and conditional gradient methods.
\newblock \emph{{SIAM} J. Optim.}, 25\penalty0 (1):\penalty0 115--129, 2015.

\bibitem[Barzilai and Borwein(1988)]{BarzBorw:1988}
J.~Barzilai and J.~M. Borwein.
\newblock Two-point step size gradient methods.
\newblock \emph{{IMA} J. Numer. Anal.}, 8:\penalty0 141--148, 1988.

\bibitem[Becker et~al.(2011)Becker, Cand\`es, and Grant]{BeckBobCandGrant:2011}
S.~Becker, E.~J. Cand\`es, and M.~Grant.
\newblock Templates for convex cone problems with applications to sparse signal
  recovery.
\newblock \emph{Math. Program. Comp.}, 3:\penalty0 165--218, 2011.

\bibitem[Bertsekas(1999)]{Bert:1999}
D.~P. Bertsekas.
\newblock \emph{Nonlinear Programming}.
\newblock Athena Scientific, Belmont, MA, second edition, 1999.

\bibitem[Bertsekas(2015)]{bertsekas2015}
D.~P. Bertsekas.
\newblock \emph{Convex optimization algorithms}.
\newblock Athena Scientific, Belmont, MA, 2015.

\bibitem[Cand\`{e}s and Li(2014)]{Candes:2014:SQE:2673201.2673259}
E.~J. Cand\`{e}s and X.~Li.
\newblock Solving quadratic equations via {PhaseLift} when there are about as
  many equations as unknowns.
\newblock \emph{Found. Comput. Math.}, 14\penalty0 (5):\penalty0 1017--1026,
  Oct. 2014.
\newblock ISSN 1615-3375.

\bibitem[Cand\`es et~al.(2012)Cand\`es, Strohmer, and
  Voroninski]{candes2012phaselift}
E.~J. Cand\`es, T.~Strohmer, and V.~Voroninski.
\newblock Phaselift: Exact and stable signal recovery from magnitude
  measurements via convex programming.
\newblock \emph{Commun. Pur. Appl. Ana.}, 2012.

\bibitem[Cand{\`e}s et~al.(2013)Cand{\`e}s, Eldar, Strohmer, and
  Voroninski]{Candes:2013}
E.~J. Cand{\`e}s, Y.~C. Eldar, T.~Strohmer, and V.~Voroninski.
\newblock Phase retrieval via matrix completion.
\newblock \emph{SIAM J. Imaging Sci.}, 6\penalty0 (1):\penalty0 199--225, 2013.

\bibitem[Cand{\`e}s et~al.(2015)Cand{\`e}s, Li, and Soltanolkotabi]{7029630}
E.~J. Cand{\`e}s, X.~Li, and M.~Soltanolkotabi.
\newblock Phase retrieval via {W}irtinger flow: Theory and algorithms.
\newblock \emph{{IEEE} Trans. Inform. Theory}, 61\penalty0 (4):\penalty0
  1985--2007, April 2015.

\bibitem[d'Aspremont and Karoui(2014)]{doi:10.1137/12088728X}
A.~d'Aspremont and N.~E. Karoui.
\newblock A stochastic smoothing algorithm for semidefinite programming.
\newblock \emph{{SIAM} J. Optim.}, 24\penalty0 (3):\penalty0 1138--1177, 2014.

\bibitem[Fazel(2002)]{Fazel:2002}
M.~Fazel.
\newblock \emph{{Matrix rank minimization with applications}}.
\newblock PhD thesis, Elec. Eng. Dept, Stanford University, 2002.

\bibitem[Freund(1987)]{freund:1987}
R.~M. Freund.
\newblock Dual gauge programs, with applications to quadratic programming and
  the minimum-norm problem.
\newblock \emph{Math. Program.}, 38\penalty0 (1):\penalty0 47--67, 1987.

\bibitem[{Freund} et~al.(2015){Freund}, {Grigas}, and
  {Mazumder}]{2015arXiv151102204F}
R.~M. {Freund}, P.~{Grigas}, and R.~{Mazumder}.
\newblock {An Extended Frank-Wolfe Method with ''In-Face'' Directions, and its
  Application to Low-Rank Matrix Completion}.
\newblock \emph{ArXiv e-prints}, Nov. 2015.

\bibitem[Friedlander et~al.(2014)Friedlander, Mac\^edo, and
  Pong]{FriedlanderMacedoPong:2014}
M.~P. Friedlander, I.~Mac\^edo, and T.~K. Pong.
\newblock Gauge optimization and duality.
\newblock \emph{{SIAM} J. Optim.}, 24\penalty0 (4):\penalty0 1999--2022, 2014.

\bibitem[Harrison(1993)]{Harrison:93}
R.~W. Harrison.
\newblock Phase problem in crystallography.
\newblock \emph{J. Opt. Soc. Am. A}, 10\penalty0 (5):\penalty0 1046--1055, May
  1993.

\bibitem[Hazan(2008)]{hazan2008sparse}
E.~Hazan.
\newblock Sparse approximate solutions to semidefinite programs.
\newblock In \emph{LATIN 2008: Theoretical Informatics}, pages 306--316.
  Springer, 2008.

\bibitem[Helmberg and Rendl(2000)]{HelmbergRendl:2000}
C.~Helmberg and F.~Rendl.
\newblock A spectral bundle method for semidefinite programming.
\newblock \emph{{SIAM} J. Optim.}, 10\penalty0 (3):\penalty0 673--696, 2000.

\bibitem[Helmberg et~al.(2014)Helmberg, Overton, and
  Rendl]{helmberg2014spectral}
C.~Helmberg, M.~L. Overton, and F.~Rendl.
\newblock The spectral bundle method with second-order information.
\newblock \emph{Optim. Methods Softw.}, 29\penalty0 (4):\penalty0 855--876,
  2014.

\bibitem[Larsen(2001)]{Lars:2001}
R.~M. Larsen.
\newblock Combining implicit restart and partial reorthogonalization in
  {L}anczos bidiagonalization, 2001.
\newblock \url{http://sun.stanford.edu/~rmunk/PROPACK/}.

\bibitem[Laue(2012)]{laue2012hybrid}
S.~Laue.
\newblock A hybrid algorithm for convex semidefinite optimization.
\newblock In \emph{Proc. 29th Intern. Conf. Machine Learning (ICML-12)}, 2012.

\bibitem[Lehoucq et~al.(1998)Lehoucq, Sorensen, and Yang]{lehoucq1998arpack}
R.~B. Lehoucq, D.~C. Sorensen, and C.~Yang.
\newblock \emph{{ARPACK} Users' guide: solution of large-scale eigenvalue
  problems with implicitly restarted {Arnoldi} methods}, volume~6.
\newblock SIAM, 1998.

\bibitem[Lewis(1996)]{lewis1996convex}
A.~S. Lewis.
\newblock Convex analysis on the hermitian matrices.
\newblock \emph{{SIAM} J. Optim.}, 6\penalty0 (1):\penalty0 164--177, 1996.

\bibitem[Ling and Strohmer(2015)]{Ling:2015}
S.~Ling and T.~Strohmer.
\newblock Self-calibration and biconvex compressive sensing.
\newblock \emph{CoRR}, abs/1501.06864, 2015.

\bibitem[Mohan and Fazel(2010)]{MohanFazel:2010}
K.~Mohan and M.~Fazel.
\newblock Reweighted nuclear norm minimization with application to system
  identification.
\newblock In \emph{American Control Conference (ACC), 2010}, pages 2953--2959,
  June 2010.

\bibitem[Nesterov(2009)]{nesterov2009unconstrained}
Y.~Nesterov.
\newblock Unconstrained convex minimization in relative scale.
\newblock \emph{Math. Op. Res.}, 34\penalty0 (1):\penalty0 180--193, 2009.

\bibitem[Nesterov(2015)]{nesterov2015complexity}
Y.~Nesterov.
\newblock Complexity bounds for primal-dual methods minimizing the model of
  objective function.
\newblock Tech. rep., Center for Operations Research and Econometrics, Feb.
  2015.

\bibitem[Netrapalli et~al.(2013)Netrapalli, Jain, and Sanghavi]{NIPS2013_5041}
P.~Netrapalli, P.~Jain, and S.~Sanghavi.
\newblock Phase retrieval using alternating minimization.
\newblock In C.~Burges, L.~Bottou, M.~Welling, Z.~Ghahramani, and
  K.~Weinberger, editors, \emph{Advances in Neural Inform. Proc. Sys. 26},
  pages 2796--2804, 2013.

\bibitem[Nocedal and Wright(2006)]{NoceWrig:2006}
J.~Nocedal and S.~J. Wright.
\newblock \emph{Numerical Optimization}.
\newblock Springer, New York, second edition, 2006.

\bibitem[Overton(1992)]{doi:10.1137/0802007}
M.~L. Overton.
\newblock Large-scale optimization of eigenvalues.
\newblock \emph{{SIAM} J. Optim.}, 2\penalty0 (1):\penalty0 88--120, 1992.

\bibitem[Recht et~al.(2010)Recht, Fazel, and Parrilo]{RFP:2010}
B.~Recht, M.~Fazel, and P.~A. Parrilo.
\newblock Guaranteed minimum-rank solutions of linear matrix equations via
  nuclear norm minimization.
\newblock \emph{{SIAM} Rev.}, 52\penalty0 (3):\penalty0 471--501, 2010.

\bibitem[Richt{\'a}rik(2011)]{richtarik2011improved}
P.~Richt{\'a}rik.
\newblock Improved algorithms for convex minimization in relative scale.
\newblock \emph{{SIAM} J. Optim.}, 21\penalty0 (3):\penalty0 1141--1167, 2011.

\bibitem[Rockafellar(1970)]{Roc70}
R.~T. Rockafellar.
\newblock \emph{Convex Analysis}.
\newblock Princeton University Press, Princeton, 1970.

\bibitem[Strohmer(2013)]{StrohmerPrivate:2013}
T.~Strohmer.
\newblock Personal communication, December 2013.

\bibitem[Waldspurger et~al.(2015)Waldspurger, d'Aspremont, and
  Mallat]{Waldspurger:2015}
I.~Waldspurger, A.~d'Aspremont, and S.~Mallat.
\newblock Phase recovery, maxcut and complex semidefinite programming.
\newblock \emph{Math. Program.}, 149:\penalty0 47--81, February 2015.

\bibitem[{White} et~al.(2015){White}, {Sanghavi}, and
  {Ward}]{2015arXiv150607868W}
C.~D. {White}, S.~{Sanghavi}, and R.~{Ward}.
\newblock {The local convexity of solving systems of quadratic equations}.
\newblock \emph{ArXiv e-prints}, June 2015.

\bibitem[Zhang and Hager(2004)]{doi:10.1137/S1052623403428208}
H.~Zhang and W.~Hager.
\newblock A nonmonotone line search technique and its application to
  unconstrained optimization.
\newblock \emph{{SIAM} J. Optim.}, 14\penalty0 (4):\penalty0 1043--1056, 2004.

\end{thebibliography}

\end{document}